\documentclass[10pt,a4paper]{article}
\usepackage[latin1]{inputenc}
\usepackage{amsmath}
\usepackage{amsfonts}
\usepackage{amssymb}
\usepackage{graphicx}

\usepackage{amsthm}
\theoremstyle{plain}
\usepackage{tikz}
\usetikzlibrary{cd}
\usepackage{mathtools}

\usepackage{hyperref}

\newtheorem{thm}{Theorem}
\newtheorem{definition}[thm]{Definition}
\newtheorem{cor}[thm]{Corollary}
\newtheorem{lem}[thm]{Lemma}
\newtheorem{rem}[thm]{Remark}
\newtheorem{prop}[thm]{Proposition}
\newtheorem{ex}[thm]{Example}

\newtheorem{thmintro}{Theorem}
\newtheorem{propintro}[thmintro]{Proposition}

\newtheorem{corintro}[thmintro]{Corollary}
\newtheorem{defintro}[thmintro]{Definition}

\newcommand{\Q}{\mathbb{Q}}

\newcommand{\R}{\mathbb{R}}
\newcommand{\C}{\mathbb{C}}
\newcommand{\Z}{\mathbb{Z}}
\newcommand{\A}{\mathbb{A}}

\newcommand{\cA}{\mathcal{A}}
\newcommand{\cS}{\mathcal{S}}

\newcommand{\cR}{\mathcal{R}}

\newcommand{\V}{\mathcal{V}}

\newcommand{\ord}{\operatorname{ord}}
\newcommand{\im}{\operatorname{im}}
\newcommand{\del}{\partial}
\newcommand{\delbar}{\bar{\partial}}

\newcommand{\Id}{\operatorname{Id}}
\newcommand{\Pro}{\mathbb{P}}

\DeclareMathOperator{\gr}{gr}
\DeclareMathOperator{\sgn}{sgn}
\DeclareMathOperator{\mult}{mult}
\DeclareMathOperator{\DC}{DC}

\DeclareMathOperator{\Hom}{Hom}
\mathtoolsset{centercolon=true}
\newcommand{\Cdot}{{\raisebox{-0.7ex}[0pt][0pt]{\scalebox{2.0}{$\cdot$}}}}
\mathchardef\mhyphen="2D

\let\oldabstract\abstract
\let\oldendabstract\endabstract
\makeatletter
\renewenvironment{abstract}
{%
	{\list{}{\addtolength{\leftmargin}{2em} 
			\listparindent 0em%
			\itemindent    \listparindent%
			\rightmargin   \leftmargin%
			\parsep        \z@ \@plus\p@}%
		\item\relax}%
	{\endlist}%
	\oldabstract}
{\oldendabstract}
\makeatother

\usepackage{authblk}
\date{}
\title{On the Structure of Double Complexes\footnote{MSC Subject Classifiers: 18G40, 32C35, 32Q99}}
\author{Jonas Stelzig}
\begin{document}
		\maketitle
	\begin{abstract}
We study consequences and applications of the folklore statement that every double complex over a field decomposes into so-called squares and zigzags. This result makes questions about the associated cohomology groups and spectral sequences easy to understand. We describe a notion of `universal' quasi-isomorphism and the behaviour of the decomposition under tensor product and compute the Grothendieck ring of the category of bounded double complexes over a field with finite cohomologies up to such quasi-isomorphism (and some variants). \\
Applying the theory to the double complexes of smooth complex valued forms on compact complex manifolds, we obtain a Poincar\'e duality for higher pages of the Frölicher spectral sequence, construct a functorial three-space decomposition of the middle cohomology, give an example of a map between compact complex manifolds which does not respect the Hodge filtration strictly, compute the Bott-Chern and Aeppli cohomology for Calabi-Eckmann manifolds, introduce new numerical bimeromorphic invariants, show that the non-Kählerness degrees are not bimeromorphic invariants in dimensions higher than three and that the $\partial\overline{\partial}$-lemma and some related properties are bimeromorphic invariants if, and only if, they are stable under restriction to complex submanifolds. 
\end{abstract}

\section*{Introduction}
Double complexes are linear-algebraic objects which arise in many situations in algebra and geometry. Possibly the most prominent example is the double complex $\cA_X$ of $\C$-valued forms on a complex manifold $X$. Other examples include the space of forms on a manifold carrying a pair of transverse foliations (c.f. Klingler's recent work on the Chern conjecture \cite{klingler_cherns_2017}) and, more generally, double complexes appearing in the construction of injective (or acyclic) resolutions of simple complexes.\\

The present article is concerned with applications of the structure theory of double complexes $A$ over a field $K$ which are bounded, i.e., have nonzero components only in finitely many bidegrees. This is the case in many of the examples mentioned before. The applications will roughly fall in two classes: Algebraic ones, dealing with the general theory of double complexes and objects associated with them, and complex-geometric ones, treating the specific double complexes $\cA_X$.\footnote{The case of a bifoliation will be treated in forthcoming work.} We will borrow notation common in the latter context also for the general discussion.\\

Associated with a bounded double complex, there are several different cohomology theories (in the particular case of $A=\cA_X$ known as de Rham, Dolbeault, Bott-Chern, Aeppli) which are related by various maps and the Frölicher spectral sequences (cf. \cite{angella_cohomological_2013}):
\[\begin{tikzcd}\tag{H(A)}\label{cohomology diagram}
&H_{BC}^{p,q}(A)\ar[ld]\ar[d]\ar[rd]&\\
H_{\del_1}^{p,q}(A)\ar[rd]\arrow[Rightarrow]{r}&(H_{dR}^{p+q}(A),F_1,F_2)\ar[d]&H^{p,q}_{\del_2}(A)\ar[ld]\arrow[Rightarrow]{l}\\
&H_A^{p,q}(A)
\end{tikzcd}\]
If this diagram degenerates (i.e., the spectral sequences degenerate on the first page and for all $p,q\in\Z$, the diagonal maps are isos and the vertical ones are injective, resp. surjective), the double complex is said to satisfy the $\del_1\del_2$-lemma. A well-known result by Deligne, Griffiths, Morgan and Sullivan \cite{deligne_real_1975} states that for $A$ a bounded double complex of $K$-vector spaces this is the case if, and only if, $A$ is isomorphic to a direct sum of double complexes of the following two kinds:
\begin{itemize}
	\item \textbf{Squares} \[
	\begin{tikzcd}
	K\ar[r,"-\Id"]& K\\
	K\arrow[u,"\Id"]\ar[r,"\Id"]& K\arrow[u,swap,"\Id"]
	\end{tikzcd}
	\]
	\item and \textbf{dots}, i.e., $1$-dimensional complexes, being necessarily concentrated in a single bidgree, with all maps equal to zero.
\end{itemize}

A somewhat lesser known theorem, going back at least to Khovanov, is the starting point of this article. It states that a similar decomposition actually holds for any bounded double complexes over fields. More precisely:
\begin{thmintro}[Theorem \ref{decomposition of double complexes}]\label{intro: decomposition dc}
	For any bounded double complex $A$ over a field $K$ there exist unique cardinal numbers $\mult_S(A)$ and a (non-functorial) isomorphism $A\cong\bigoplus_{S}S^{\oplus \mult_S(A)}$, where $S$ runs over
	\begin{itemize}
		\item \textbf{Squares} 
		\item and \textbf{zigzags} 
		\begin{center}
			\begin{tabular}{ccccccccc}
				$K$
				&,&
				${\begin{tikzcd}
					K\\
					K\ar[u,"\Id"]
					\end{tikzcd}}$&,&
				${\begin{tikzcd}
					K\ar[r,"\Id"]&K
					\end{tikzcd}}$
				&,&
				${\begin{tikzcd}
					K&\\
					K\ar[u,"\Id"]\ar[r,"\Id"]&K
					\end{tikzcd}}$&,&$\cdots$
			\end{tabular}
		\end{center}
	\end{itemize}
\end{thmintro}

An elementary proof of this theorem is given in section \ref{sec: Decomposition}. In the first part of section \ref{sec: Cohomologies and Multiplicities}, we study the cohomology diagram \ref{cohomology diagram} from the point of view of Theorem \ref{intro: decomposition dc}. Summarised briefly, squares do not contribute to any cohomology, even length zigzags correspond to differentials in the Frölicher spectral sequences, while odd length zigzags correspond to de Rham classes (the length of a zigzag being defined as the number of nonzero components). Bott-Chern and Aeppli cohomology count top right, resp. bottom left, corners in the zigzags. In particular, one obtains:

\begin{corintro}[Prop. \ref{multiplicities and cohomology} and Lem. \ref{Dolbeault, Bott Chern and Aeppli via zigzags}]\label{intro: cohomology-dimensions via zigzags}
	For any bounded double complex $A$ over a field, the dimensions of $H^d_{dR}(A)$, $H^{p,q}_{\del_1}(A)$, $H^{p,q}_{\del_2}(A)$, $H_{BC}^{p,q}(A)$, $H_A^{p,q}(A)$ are linear combinations (with coefficients $1$ or $0$) of the numbers $\mult_Z(A)$ for zigzags $Z$.
\end{corintro}

Conversely, if one is willing to take into account the two filtrations on the de Rham cohomology, one may reconstruct all multiplicities of zigzags from the de Rham cohomology and the Frölicher spectral sequences. In fact, the \textbf{refined Betti numbers} $b_d^{p,q}:=\dim \gr_{F_1}^p\gr_{F_2}^qH^{d}_{dR}(A)$ (which sum up to the $d$-th Betti number $b_d:=\dim H^d_{dR}(A)$ and specialise to the Hodge numbers in case the double complex satisfies the $\del_1\del_2$-lemma) encode the multiplicities of individual odd zigzags. In particular, one obtains the following characterisation of degeneration of the Frölicher spectral sequence(s) and Hodge-structures on de Rham cohomology:

\begin{thmintro}[Cor. \ref{generalised DGMS}]\label{intro: generalised DGMS}
	Let $A$ be a bounded double complex over a field. 
	\begin{enumerate}
		\item The two Frölicher spectral sequences degenerate on the $r$-th page if and only if the length of all even zigzags appearing in some (any) decomposition as above is smaller than $2r$.
		\item There is a pure Hodge structure of weight $k$ on the total cohomology in degree $d$, i.e.
		\[
		H^d_{dR}(A)=\bigoplus_{p+q=k} (F_1^p\cap F_2^q)H^d_{dR}(A)
		\]
		if and only if all zigzags contributing to $H^d_{dR}(A)$ have length $2|d-k|+1$ and are concentrated in total degree $d$ and $d+\operatorname{sgn}(d-k)$.		
	\end{enumerate}
	If the involved quantities are finite, then the first point is equivalent to the equality $\sum_{p,q\in\Z} \dim E_r^{p,q}=\sum_{k\in\Z} b_k$ and the second to $b_d=\sum_{p+q=k}b_d^{p,q}$.
\end{thmintro}

Together, Theorems \ref{intro: decomposition dc} and \ref{intro: generalised DGMS} generalise the result of \cite{deligne_real_1975} to arbitrary (bounded) double complexes. Apart maybe from the consideration of refined Betti numbers, they have been known to some experts (\cite{khovanov_spectral_2007} \cite{qi_quiver_2010} \cite{kuperberg_notitle_2010}, \cite{speyer_notitle_2012}, \cite{megy_three_2014} and \cite{khovanov_faithful_2020}), but there used to be no reference available. They have also been used in non-Kähler geometry, albeit in lack of reference as heuristics only (\cite{angella_bott-chern_2015}, \cite{angella_hodge_2018}, \cite{angella_cohomological_2013}). As we exemplify in Corollary \ref{non-deldelbar-degrees}, the exposition given here allows to turn these heuristics into actual proofs.\\

With this background at hand, let us survey the main applications: First, the following refined notion of quasi-isomorphism is particularly well-behaved (see also \cite{cirici_model_2019} and \cite{neisendorfer_dolbeault_1978} for studies of different notions of quasi-isomorphism from a rational homotopy theoretic point of view.):

\begin{defintro}[Def. \ref{def: quiso}]
	A morphism of bounded double complexes $A\longrightarrow B$ is called an \textbf{$E_r$-isomorphism} ($r\in\Z_{>0}\cup\{\infty\}$) if it induces an isomorphism on the $r$-th page of both Frölicher spectral sequences. Write $A\simeq_r B$ if such a morphism exists.
\end{defintro}

\begin{propintro}[Prop. \ref{characterisation of E_r-isomorphisms}, Prop. \ref{linear functors and quasi-isomorphisms} and Cor. \ref{E_1 q.i. implies BC q.i.}]\label{intro: E_r-isomorphisms} Let $A,B$ be bounded double complexes over a field $K$.
	\begin{itemize}
		\item $A\simeq_1 B$ if and only if $\mult_S(A)=\mult_S(B)$ for all zigzags
		\item An $E_1$-isomorphism $A\longrightarrow B$ also induces an isomorphism in Bott-Chern and Aeppli cohomology and more generally under any linear functor from bounded double complexes to $K$-vector spaces that maps squares to $0$.
	\end{itemize} 
\end{propintro}

We also give a version for $E_r$-isomorphisms. In particular, $\simeq_r$ is an equivalence relation on bounded double complexes over fields and the equivalence classes of $\simeq_1$ contain all cohomological information.\\

In many situations, e.g. when considering the product of two complex manifolds, one is lead to consider the tensor product of two double complexes. In section \ref{sec: Grothendieck rings} we therefore describe the behavior of the decomposition under this operation. As an application of theoretical nature, we compute the rings $\cR_r$ of formal linear combinations of $E_r$-isomorphism classes of bounded double complexes with finite dimensional $E_r$-page (forcing sum and product in the ring to be induced by direct sum and tensor product). The main result is:

\begin{thmintro}[Thm. \ref{Grothendieck ring of dc}]
	The ring $\cR_\infty\cong \Z[U^{\pm 1}, R^{\pm 1}, L^{\pm 1}]$ is a Laurent polynomial ring in three variables. The ring $\cR_1$ is a (still infitely generated) quotient of $\cR_\infty[\{X_l\}_{l\geq1},\{Y_l\}_{l\geq 1}]$, where the two sets of generators satisfy $X_l\cdot Y_{l'}=0$ (and all further relations are given explicitly).
\end{thmintro}

We also state variants where we only consider first-quadrant double complexes or only complexes satisfying the $\del_1\del_2$-lemma (see Theorem \ref{Grothendieck ring of upper quadrant dc}).\\

In section \ref{sec: complex manifolds}, we apply the theory to our main example: The double complex $\cA_X$ of $\C$-valued forms on a compact complex manifold $X$, which we call the \textbf{Dolbeault double complex}.\\

First we review some foundational results on this complex in the light of the theory sketched above, which already allows to strengthen some of them significantly. For example, one may read Serre duality as the statement that the map from $\cA_X$ to its dual complex is an $E_1$-quasi isomorphism. From this one gets

\begin{corintro}
Let $X$ be a compact complex manifold of dimension $n$. There is a canonical isomorphism 
\[
E_r^{p,q}(X)\cong (E_r^{n-p,n-q}(X))^\vee,
\]
where $E_r(X)$ denotes the $r$-th page of the Frölicher spectral sequence. A similar formula holds if one replaces $E_r$ with any linear functor from bounded double complexes to vector spaces that maps squares to $0$.
\end{corintro}

Next, we compute $\cA_X$ up to $E_1$-isomorphism in various cases: For compact Kähler manifolds, for most Calabi-Eckmann manifolds (building on a result by Borel) and for a particular nilmanifold. As a byproduct, we obtain: 

\begin{corintro}
Let $M_{u,v}$ be $S^{2u+1}\times S^{2v+1}$ with one of the Calabi-Eckmann complex structures and assume $u<v$. Then

\[
\dim H_{BC}^{p,q}(M_{u,v})=\begin{cases}
2 & \text{if }u\geq 1\text{ and } (p,q)\in\{(1,1),...,(u,u)\}\\
1 & \text{if } (p,q)\in\{(0,0),(u+1,u+1),(u+v+1,u+v+1)\}\\
& \text{or } (p,q)\in\{(v,v+1),...,(u+v,u+v+1)\}\\
&\text{or } (p,q)\in\{(v+1,v),...(u+v+1,u+v)\}\\
0 & \text{else}
\end{cases}
\]
This also yields a formula for Aeppli-cohomology by duality.
\end{corintro}

\begin{propintro}
	There is a compact complex $3$-fold $X$ and a holomorphic map $\varphi: X\longrightarrow X$ s.t. $\varphi^*$ does not respect the Hodge filtration on the de Rham cohomology strictly.
\end{propintro}

Contrast the second statement with the case of Kähler (or $\del\delbar$-)manifolds, where any geometrically induced morphism automatically respects the Hodge filtration strictly for linear algebraic reasons. Although likely the expected behaviour for maps of general compact complex manifolds, this appears to be the first example in the literature.\\

In the final part, we turn to questions about bimeromorphic invariants: In earlier work \cite{stelzig_double_2019}, $\cA_X$ was computed up to $E_1$-isomorphism for projective bundles, modifications and blow-ups. Combined
with the theory developed here, one may obtain statements about bimeromorphic invariants (see also \cite{yang_bottchern_2020}, \cite{rao_dolbeault_2019} and \cite{angella_note_2020} for other works in this direction):

\begin{propintro}[Cor. \ref{bimeromorphically invariant multiplicities}]
	For $X$ a compact complex manifold of dimension $n$, the multiplicites in $\cA_X$ of all zigzags which have a nonzero component in the region 
	\[
	\square:=\{(p,q)\in\Z^2_{\geq 0}\mid p\in\{0,n\}\text{ or }q\in\{0,n\}\}
	\] 
	are bimeromorphic invariants. The same holds for the multiplicities of zigzags which are not dots and have a nonzero component in bidegree $(1,1)$, $(n-1,1)$, $(1,n-1)$ or $(n-1,n-1)$.
\end{propintro}

It is known that for a compact complex manifold $X$, the dimensions of $H_{BC}^{p,0}(X)$, $H_{\delbar}^{0,p}(X)$ and $H_{\delbar}^{p,0}(X)$ are bimeromorphic invariants. The invariants here are finer: The multiplicites of any zigzag hitting the boundary region (generalising the previously known invariants) and the multiplicities of non-dot zigzags in the vicinity of the corners of $\cA_X$. We stress that the latter are numerical invariants not concentrated in degrees $(0,p)$ or $(p,0)$.\\

Let us say a property $(P)$ of compact complex manifolds is a bimeromorphic invariant if for two bimeromorphic compact complex manifolds one satisfies $(P)$ if and only if the other one does. Similarly, $(P)$ is said to be stable under restriction if any complex submanifold of a compact complex manifold satisfying $(P)$ also satisfies $(P)$.\\

It has been proved that that the $\del\delbar$-lemma is a bimeromorphic invariant in dimensions up to three and degeneration of the Frölicher spectral sequence at the first page is a bimeromorphic invariant in dimensions up to four (see \cite{yang_bottchern_2020}, \cite{rao_dolbeault_2019} and also \cite{angella_note_2020}). This admits the following generalisation:

\begin{thmintro}[Cor. \ref{bimeromorphical invariance and submanifold inheritance}]
		The following properties are bimeromorphic invariants of compact complex manifolds if and only if they are stable under restriction.
	\begin{itemize}
		\item The Frölicher spectral sequence degenerates at stage $\leq r$.
		\item The $k$-th de Rham cohomology groups satisfy a Hodge decomposition \[H_{dR}^k=\bigoplus_{p+q=k} (F^p\cap \overline{F}^q)H_{dR}^k\] for all $k$.
		\item The $\del\delbar$-lemma holds.
	\end{itemize}
\end{thmintro}
In \cite{angella_$partialoverlinepartial$_2013}, the non-Kählerness-degrees 
\[
\Delta^k(X):=\sum_{p+q=k}\dim H^{p,q}_{BC}(X)+\dim H^{p,q}_A(X)-2\dim H^k_{dR}(X)
\]
were introduced and shown to be non-negative with vanishing being equivalent to the $\del\delbar$-lemma. We generalise this to arbitrary bounded double complexes for which the involved quantities are finite by a method building on Theorem \ref{intro: decomposition dc} and sketched as a heuristic in \cite{angella_bott-chern_2015}. In \cite{yang_bottchern_2020}, it was shown that these are bimeromorphic invariants in dimensions up to three and asked whether this was true in higher dimensions. We settle this question by proving:
\begin{thmintro}[Cor. \ref{non-Kaehlerness-degrees of blowups}]
	Given a blowup $\widetilde{X}$ of a compact complex manifold $X$ along a submanifold $Z$, the non-Kählerness degrees satisfy 
	\[
	\Delta^k(\widetilde{X})\geq\Delta^k(X)
	\]
	and equality holds for $k=0,1,2,2n-2,2n-1,2n$ (and $k=3$ if $n=3$). Equality holds for all $k$ if and only if $Z$ satisfies the $\del\delbar$-lemma. 
\end{thmintro}
By considering any complex manifold that admits a non $\del\delbar$-submanifold, one sees:
\begin{corintro}
	The numbers $\Delta^k(X)$ are generally not bimeromorphic invariants in dimensions $n\geq 4$.
\end{corintro}

\textbf{Added in Proof:} After submission of the current manuscript, the article \cite{khovanov_faithful_2020} appeared, which, following the unpublished note \cite{khovanov_spectral_2007}, contains a different proof of Theorem \ref{intro: decomposition dc} and a discussion similar to the first part of section \ref{sec: Cohomologies and Multiplicities}.

\section{Decomposing double complexes}\label{sec: Decomposition}
\textbf{Notations and conventions:} The letter $K$ will always denote a field. By a double complex (sometimes also called bicomplex) over $K$, we mean a bigraded $K$-vector space $A=\bigoplus_{p,q\in\Z}A^{p,q}$ with two endomorphisms $\del_1,\del_2$ of bidegree $(1,0)$ and $(0,1)$ that satisfy the `boundary condition' $\del_i\circ\del_i=0$ for $i=1,2$ and anticommute, i.e., $\del_1\circ\del_2+\del_2\circ\del_1=0$.\footnote{The \textit{anti-} not essential. In fact, replacing $\del_1$ by $\del_1'$ defined by $(\del_1')^{p,q}:=(-1)^{p}\del_1^{p,q}$ we can pass to a commutative double complex (satisfying $\del_1\circ\del_2=\del_2\circ\del_1$) and vice versa.} We write $\del_1^{p,q}$ for the map from $A^{p,q}$ to  $A^{p+1,q}$ induced by restriction and similarly for $\del_2^{p,q}$. We always assume double complexes to be bounded, i.e., $A^{p,q}=0$ for almost all $(p,q)\in\Z^2$ and denote by $\operatorname{DC}_K^b$ the category of bounded double complexes over $K$ and $K$-linear maps respecting the grading and the $\del_i$. If no confusion is likely to result, we say \textit{complex} instead of \textit{double complex over a field $K$}. \\

The following is a standard definition:

\begin{definition}
A (nonzero) double complex $A$ is called \textbf{indecomposable} if there is no nontrivial decomposition $A=A_1\oplus A_2$.
\end{definition} 

\begin{ex} The following double complexes over $K$ are indecomposable. The drawn components are supposed to be one-dimensional and the drawn maps to be isomorphisms, while all components and maps not drawn are zero.\\
	 
\textbf{Squares} \[
\begin{tikzcd}
A^{p-1,q}\arrow[r,"\del_1"]& A^{p,q}\\
A^{p-1,q-1}\arrow[u,"\del_2"]\ar[r,"\del_1"]& A^{p,q-1}\arrow[u,swap,"\del_2"]
\end{tikzcd}
\]
and \textbf{zigzags} 
\begin{center}
	\begin{tabular}{cccccccc}
		$A^{p,q}$
		&,&
		${\begin{tikzcd}
			A^{p,q+1}\\
			A^{p,q}\arrow{u}{\del_2}
			\end{tikzcd}}$&,&
		${\begin{tikzcd}
			A^{p,q}\ar[r,"\del_1"]&A^{p+1,q}
			\end{tikzcd}}$
		&,&
			${\begin{tikzcd}
			A^{p,q+1}&\\
			A^{p,q}\ar[u,"\del_2"]\ar[r,"\del_1"]&A^{p,q+1}
			\end{tikzcd}}$&,
	\end{tabular}\\
\bigskip
\begin{tabular}{ccccc}

	${\begin{tikzcd}
		A^{p-1,q}\ar[r,"\del_1"]&A^{p,q}\\
		&A^{p,q-1}\ar[u,"\del_2"]
		\end{tikzcd}}$&,&
	${\begin{tikzcd}
		A^{p,q}\ar[r,"\del_1"]&A^{p+1,q}\\
		&A^{p+1,q-1}\ar[u,"\del_2"]\ar[r,"\del_1"]& A^{p+2,q-1}
		\end{tikzcd}}$&,&$\cdots$
\end{tabular}
\end{center}
\end{ex}
For a square or a zigzag $A$, the \textbf{shape} is defined to be the set 
\[
S(A):=\{(p,q)\in\Z^2\mid A^{p,q}\neq 0\}.
\] The isomorphism class of a square or a zigzag $A$ is uniquely determined by $S(A)$. If we say shape in the following, we always mean the shape of a square or a zigzag. Let us choose a section $S\mapsto C(S)$ which associates to each shape a square (resp. zigzag) of this shape. For concreteness, one may always choose all nonzero components to be $K$ and all nonzero differentials to be $\pm\Id$.\\

\begin{thm}\label{decomposition of double complexes}
	For every bounded double complex $A$ over $K$, there exist unique cardinal numbers $\mult_S(A)$ and a (non-unique) isomorphism
	\[
	A\cong \bigoplus C(S)^{\oplus\mult_S(A)},
	\]
	where $S$ runs over the set of all shapes of squares and zigzags. In particular, every indecomposable complex is isomorphic to a square or a zigzag.
\end{thm}

It will be convenient to call \textbf{elementary complex} a complex $T$ which is a direct sum of squares or zigzags of a single isomorphism type (i.e. a summmand in the big sum above). The shape $S(T)$ (defined as before) coincides with the shape of any indecomposable component of $T$. Elementary complexes are intrinsically characterised as those complexes in which every map is an isomorphism or zero and whose undirected support graph is connected. 

\begin{proof}
The main strategy is to define a filtration on an arbitrary double complex which behaves functorially and s.t. the associated graded pieces are less complicated than the original complex. Then one shows that the filtration splits (giving existence of a decomposition into the less complicated pieces) and uses that a filtered isomorphism of filtered double complexes induces an isomorphism of the associated graded pieces (giving uniqueness). This process is repeated several times until the associated graded pieces are elementary complexes.\\

Consider the functorial ascending filtration $W_\Cdot$ on $A$ given in total degree $k$ by the subcomplex generated by all components in total degree $\leq k$, i.e.,
\[
(W_kA)^{p,q}=\begin{cases}
A^{p,q}&\text{if }p+q\leq k\\
(\im\del_1+\im\del_2)^{p,q}&\text{if }p+q=k+1\\
(\im\del_1\circ\del_2)^{p,q}&\text{if }p+q=k+2\\
\{0\}&\text{else.}
\end{cases}
\]
The filtration $W_\Cdot$ splits: Choose (in each degree) complements $D_{\del_1}\oplus\im\del_1\del_2=\im\del_1$ and $D_{\del_2}\oplus\im\del_1\del_2=\im\del_2$ with $\im\del_1+\im\del_2=\im\del_1\del_2\oplus (D_{\del_1}+D_{\del_2})$. One verifies that $(\im\del_1+\im\del_2)+(\del_1^{-1}D_{\del_1}\cap\del_2^{-1}D_{\del_2})=A$ and so one may pick degree-wise a complement $A=(\im\del_1+\im\del_2)\oplus C$ with  $C\subseteq(\del_1^{-1}D_{\del_1}+\del_2^{-1}D_{\del_2})$. Defining $B_k$ to be the subcomplex generated in total degree $k$ by $C$, i.e.
\[
B_k^{p,q}:=\begin{cases}
C^{p,q}&\text{if }p+q=k\\
(D_{\del_1}+D_{\del_2})^{p,q}&\text{if }p+q=k+1\\
(\im\del_1\circ\del_2)^{p,q}&\text{if }p+q=k+2\\
\{0\}&\text{else,}
\end{cases}
\]
one has $W_\Cdot A=\bigoplus_{k\leq \Cdot}B_k$.\\

Given any decomposition into elementary complexes $A\cong\bigoplus T_i$, it induces an isomorphism
\[
\gr_k^W\!\!A\cong\bigoplus_{\substack{T_i \text{ generated}\\\text{in degree } k}} T_i.
\]
We are thus reduced to the case that $A$ is generated in a single total degree $k$.\\

Given $A$ generated in a single total degree $k$, for all $p,q$ with $p+q=k$, set $K^{p,q}:=(\ker{\del_1\circ\del_2})^{p,q}$ and define $K$ to be the subcomplex generated by the $K^{p,q}$. Thus, we have a two-step filtration $W'$:
\[
W'_0=\{0\}\subseteq K\subseteq A=W'_2
\]
Given any decomposition $A\cong\bigoplus T_i$ into elementary complexes with distinct support, there are isomorphisms
\[
\gr_1^{W'}\!\!A\cong\bigoplus_{S(T_i) \text{ zigzag shape}}T_i \qquad \gr_2^{W'}\!\!A\cong\bigoplus_{S(T_i) \text{ square shape}}T_i.
\]
Thus, we are reduced to check uniqueness in the two cases that the complex is generated in degree $k$ and either all the $T_i$ are direct sums of zigzags or all the $T_i$ are direct sums of squares. In this last case, the $T_i$ with base in $(p,q)$ has the intrinsic definition as the subcomplex generated by $A^{p,q}$, so the decomposition is unique. Also, $W'_\Cdot$ splits: In fact, for every $p,q\in\Z$ with $p+q=k$, choose a complement $S^{p,q}$ s.t. $A^{p,q}=K^{p,q}\oplus S^{p,q}$. Let $S$ denote the subcomplex generated by the $S^{p,q}$. By construction, $S$ splits uniquely as a direct sum of squares generated in degree $k$ (the subcomplexes generated by the $S^{p,q}$) and we have a direct sum decomposition
\[
A=S\oplus K.
\]

It remains to treat the case of a complex generated in a single total degree $k$ and concentrated in total degrees $k, k+1$. For such a complex the conditions $\del_i\circ\del_i=0$ and $\del_1\circ\del_2+\del_2\circ\del_1=0$ are vacuous and, after relabeling, it is nothing than a representation of a quiver of type $\A_n$ (defined below), for which the statement needed follows from Lemma \ref{quivers of type A_n} below.
\end{proof}

Recall that a quiver of type $\A_n$ is a directed graph obtained from the following diagram
\[
\begin{tikzcd}
1\ar[r,dash]&2\ar[r,dash]&...\ar[r,dash]&n
\end{tikzcd}
\]
by assigning a direction to each dash (in the case we are interested in, in an alternating manner). A representation of such a quiver is given by assigning a vector space to each dot and a linear map to each arrow (in accordance with the specified direction).\\

A nonzero representation of a quiver of type $\A_n$ is called indecomposable if there is no nontrivial decomposition into subrepresentations. A subset $S\subseteq \underline{n}:=\{1,...,n\}$ is called connected if it is the intersection of $\underline{n}$ with some connected real interval. Given a quiver $Q$ of type $\A_n$, one obtains an indecomposable representation $I_{S}$ for every nonempty connected subset $S\subset\underline{n}$ which is $K$ on every dot in $S$ and has all possible maps the identity. For example, for the (up to relabeling unique) quiver of type $\A_2$
\[
1\longrightarrow 2
\]
the indecomposables obtained in this way are
\begin{align*}
K\longrightarrow& 0\\
K\overset{\Id}{\longrightarrow}& K\\
0\longrightarrow& K.
\end{align*}

\begin{lem}\label{quivers of type A_n}
Let $Q$ be a quiver of type $\A_n$ and $A$ a representation of $Q$. There are unique (cardinal) numbers $\mult_{S}(A)$ and a (non-unique) isomorphism
\[
A\cong \bigoplus I_{S}^{\oplus \mult_{S}(A)},
\]
where $S$ runs over a all connected subsets of $\underline{n}$. In particular, each indecomposable representation $Q$ is of the form $I_{S}$.
\end{lem}
As above, we will call representations isomorphic to $I_{S}^{\oplus r}$ for some (cardinal) number $r$ elementary representations.

\begin{proof}
In the finite dimensional case, this result is due to Gabriel \cite{gabriel_unzerlegbare_1972}. It has also been studied in the context of persistent homology in \cite{carlsson_parametrized_2019}, where references for the infinite dimensional case are given: The decomposition is implied by a theorem of Auslander \cite{auslander_representation_1974} and uniqueness follows from the Krull-Schmidt-Azumaya theorem \cite{azumaya_corrections_1950}. Since \cite{auslander_representation_1974} and \cite{azumaya_corrections_1950} contain more general and technical statements than needed here, we sketch an elementary proof here for completeness, which is a quite minor adaption to the infinite dimensional case of the arguments in \cite{gabriel_unzerlegbare_1972} and the presentation in \cite{ringel_quivers_2012}:\\

First one proves the lemma `by hand' for the cases $n=1,2,3$: The case $n=1$ is trivial and the cases $n=2,3$ can be handled in a similar manner to the proof of Theorem \ref{decomposition of double complexes}, by defining a canonical filtration and construct a splitting. For example, for $n=2$, and a representation
\[
A_1\overset{\alpha}{\longrightarrow} A_2,
\]
the filtration is given by 
\[
((\ker\alpha)\longrightarrow 0)\subseteq (A_1\longrightarrow\im\alpha)\longrightarrow(A_1\longrightarrow A_2)
\]
and the splitting is constructed by choosing complements of $\ker\alpha\subseteq A_1$ and $\im\alpha\subseteq A_2$.\\

In the case $n=3$, there are, up to isomorphism, three possible quivers:
\[
\begin{tikzcd}
1&2\ar[r]\ar[l]&3&&\text{`source'}\\
1\ar[r]&2&3\ar[l]&&\text{`sink'}\\
1\ar[r]&2\ar[r]&3&&\text{`river'}
\end{tikzcd}
\]
We indicate only the filtration in the first case, leaving the splitting and the other cases to the reader. Let \[
A\overset{\alpha}{\longleftarrow}C\overset{\beta}{\longrightarrow}B
\] 
a representation of the `source' quiver. The first column of the following table is the promised filtration, whereas the second indicates the support of the associated graded in that step.
\[
\begin{tikzcd}0&0\ar[l]\ar[r]&0&&&000\\
0&\ker\alpha\cap\ker\beta\ar[l]\ar[r]&0&&&0\bullet 0\\
0&\ker\alpha\ar[l]\ar[r]&\beta(\ker\alpha)&&&0 \bullet\bullet\\
\alpha(\ker\beta)&\ker\alpha+\ker\beta\ar[l]\ar[r]&\beta(\ker\alpha)&&&\bullet\bullet0\\
\im\alpha&C\ar[l]\ar[r]&\im\beta&&&\bullet\bullet\bullet\\
\im\alpha&C\ar[l]\ar[r]&B&&& 0 0 \bullet\\
A&C\ar[l]\ar[r]&B&&&\bullet 0 0
\end{tikzcd}
\]

For general $n>3$, let $V$ be a representation of a quiver of type $\A_n$. Denote by $V|_{\{1,...,n-1\}}$ its restriction to the first $n-1$ nodes. Inductively, we can assume $V|_{\{1,...,n-1\}}\cong\bigoplus_{i=1}^rT_i$ for some (essentially unique) elementary representations $T_i$. Grouping together those $T_i$ with $T_i(n-1)=0$ and those with $T_i(n-1)\neq 0$, we obtain a decomposition
\[
V=V^-\oplus V',
\]
where $V^-(n-1)=0=V^{-1}=0$ and $V'$ is \textbf{increasing} up to degree $n-1$, i.e., if an arrow goes from $V'(i-1)$ to $V'(i)$ with $1<i\leq n-1$, it is injective, whereas it is surjective if if goes from $V'(i)$ to $V'(i-1)$. Both summands are unique up to isomorphism.\\

Similarly, denote $V'|_{\{2,...,n\}}$ the restriction of $V'$ to the subquiver given by the last $n-1$ nodes. Again, this splits by induction as a sum of elementary representations and we obtain a splitting with summands unique up to isomorphism
\[
V'=V''\oplus V^+,
\]
where $V^+(1)=0=V^+(2)$ and $V''$ is \textbf{decreasing} from degrees $2$ to $n$, i.e., for $2\leq i<n$, if a morphism goes from $V''(i)$ to $V''(i+1)$ it is surjective, and if it goes from $V''(i+1)$ to $V''(i)$, it is injective. But it is also increasing, hence all morphisms between $V''(2)$ and $V''(n-1)$ are isomorphisms and we may contract $V''$ to a representation of a quiver of type $\A_3$, where we know the statement.
\end{proof}

\begin{rem}\leavevmode
	\begin{itemize}
		\item The proof shows that Theorem \ref{decomposition of double complexes} remains true if we only assume $A$ to have \textbf{bounded antidiagonals}, i.e.: For any $k$, there are only finitely many pairs $(p,q)$ with $p+q=k$ s.t. $A^{p,q}\neq 0$. Without this condition, `infinite zigzags' may occur.
		 
		\item The consideration of cohomological invariants in the next section (Proposition \ref{multiplicities and cohomology}) yields another proof for uniqueness of the numbers $\mult_S(A)$ in Theorem \ref{decomposition of double complexes}.
		
		\item Essentially the only thing used in the proof is that one can choose complements of subvectorspaces. Therefore, one can adapt Theorem \ref{decomposition of double complexes} to double complexes in other abelian categories $\mathcal{C}$ which are semisimple in a suitable sense, e.g., characteristic $0$ representations of a finite group $G$. Isomorphism classes of indecomposable complexes are then described by pairs $(S,V)$, where $S$ is a shape and $V$ an isomorphism class of a simple object in $\mathcal{C}$ (corresponding to any nonzero component).
		
		\item After replacing $\del_1^{p,q}$ by $(-1)^p\del_1^{p,q}$, a double complex is a complex of complexes. So one might hope to get a similar statement for complexes of complexes of complexes and so on. In particular, this would also treat maps between double complexes.\\
		
		However, as was pointed out to me by L. Hille, there is no longer a discrete classification of elementary complexes. In fact, for any isomorphism $\alpha:K\cong K$ consider the following complex, where all arrows except $\alpha$ are the identity: 
		\[
		\begin{tikzpicture}
		\node (02) at (0,2) {K};
		\node (22) at (2,2) {K};
		\node (13) at (1,3) {K};
		\node (11) at (1,1) {K};
		\node (31) at (3,1) {K};
		\node (20) at (2,0) {K};
		
		\draw[->]	(02) to (13);
		\draw[->]	(11) to (13);
		\draw[->]	(11) to (31);
		\draw[->]	(20) to node [auto,swap]{$\alpha$}(31);
		\draw[line width=2mm, color=white]	(20) to (22);
		\draw[->]	(20) to (22);
		\draw[line width=2mm, color=white] 	(02) to (22);
		\draw[->] 	(02) to (22);
		
		\end{tikzpicture}
		\]
		These are pairwise nonisomorphic for different $\alpha$.\\
		
		Even worse, the indecomposables do not have to have the same dimension in every component: For example, a triple complex of the following form, where only nonzero arrows are drawn, cannot be decomposed:
		\[
		\begin{tikzpicture}[scale=2,>=stealth]
		\node (01) at (0,1) {K};
		\node (02) at (0,2) {K};
		\node (11) at (1,1) {$K^2$};
		\node (10) at (1,0) {K};
		\node (12) at (1,2) {K};
		\node (20) at (2,0) {K};
		\node (21) at (2,1) {K};
		
		\node (bc) at (3,1.5) {K};
		\node (bl) at (2,1.5) {K};
		\node (bb) at (3,0.5) {K};

		\draw[->] (11) -- (bc);
		\draw[->] (10) -- (bb);
		\draw[->] (01) -- (bl);
		
		\draw[->] (bl) -- (bc);
		\draw[->] (bb) -- (bc);
		
		\draw[->] (01) -- (02);
		\draw[->] (01) -- (11);
		\draw[->] (02) -- (12);
		\draw[line width=2mm, color=white] (11) -- (12);
		\draw[->] (11) -- (12);
		\draw[->] (10) -- (11);
		\draw[->] (11) -- (21);
		\draw[line width=2mm,color=white] (20) -- (21);
		\draw[->] (20) -- (21);
		\draw[->] (10) -- (20);

		\end{tikzpicture}
		\]
		This example also shows that there can be a map from a double complex consisting only of squares s.t. the kernel and image consist only of zigzags.
\end{itemize}
\end{rem}

\section{Cohomologies and multiplicities}\label{sec: Cohomologies and Multiplicities}
The previous chapter showed that the isomorphism type of a double complex is uniquely determined by the (cardinal) numbers $\mult_S(A)$. In this section, we show how these numbers relate to more classical cohomological invariants.\\

In all of the following, $A$ denotes a bounded double complex. We briefly recall several standard constructions:
\begin{itemize}
	
	\item The \textbf{total complex} is the simple complex given by summing up the antidiagonals: \[A_{tot}^\Cdot:=\bigoplus_{p+q=\Cdot}A^{p,q}\] with differential $d:=\del_1+\del_2$.
	
	\item The \textbf{total (or de Rham) cohomology} is the cohomology of the total complex:
	\[
	H_{dR}^k(A):=H_{tot}^k(A):=H^k(A_{tot}^\Cdot,d).
	\]
	
	\item The \textbf{row and column (or Dolbeault) cohomologies} are given by taking cohomology with respect of one of the two differentials:
	\[
	H_{\del_1}^{p,q}(A):=H^p(A^{\Cdot,q},\del_1)\text{ and }H_{\del_2}^{p,q}:=H^q(A^{p,\Cdot},\del_2).
	\] 

	\item The \textbf{filtrations by columns and rows} 
	\[
	F_1^\Cdot:=\bigoplus_{p\geq\Cdot}A^{p,q}\text{ and } F_2^\Cdot:=\bigoplus_{q\geq \Cdot}A^{p,\Cdot}
	\]
	induce filtrations on the total complex on the total cohomology. We will still denote by $F_i$ these last filtrations and call them \textbf{Hodge filtrations}. If not explicitly mentioned otherwise, in the following we will always mean the Hodge filtrations if we write $F_i$.\\
	
	The filtrations by columns and rows also induce the converging \textbf{Frölicher spectral sequences}, which compute the Hodge filtrations on the total cohomology from the column or row cohomology of the double complex:
	\[
	S_1:\qquad _1E_1^{p,q}=H^{p,q}_{\del_2}(A)\Longrightarrow \left(H^{p+q}_{dR}(A), F_1\right)
	\]
	\[
	S_2:\qquad _2E_1^{p,q}=H^{p,q}_{\del_1}(A)\Longrightarrow \left(H^{p+q}_{dR}(A), F_2\right).
	\]

	\item The\textbf{ Bott-Chern and Aeppli cohomologies}:
	\[
	H_{BC}^{p,q}(A):=\left(\frac{\ker\del_1\cap\ker\del_2}{\im\del_1\circ\del_2}\right)^{p,q}\text{ and }H^{p,q}_A(A):=\left(\frac{\ker\del_1\circ\del_2}{\im\del_1+\im\del_2}\right)^{p,q}.
	\]
	The identity induces natural maps from the Bott-Chern cohomology to row, column and total cohomology and from those three to the Aeppli-cohomology. If the induced map from Bott-Chern to Aeppli cohomology is injective for all $(p,q)\in\Z^2$, $A$ is said to \textbf{satisfy the $\del_1\del_2$-lemma}.
\end{itemize} 

We now investigate these cohomologies in detail for indecomposable double complexes. To describe the results precisely, we will have to label the possible indecomposable complexes, or rather their shapes. Even though shapes are by definition just certain subsets of $\Z^2$, when drawing them we prefer to draw the entire labeled directed support graph of any complex with the given shape.\\

\textbf{Zigzags}\\
Given any zigzag $Z$, the length of $Z$ is defined to be the number of elements in its shape $l(Z):=\#S(Z)$. We will distinguish even and odd zigzags, according to their length. A zigzag length one will also be called a dot.\\

\textbf{Even zigzags:}\\
Given a zigzag of length $l=2r$ for some integer $r\geq 1$. We denote its shape as $S^{p,q}_{i,r}$, where $(p,q)$ is the bidegree of the `starting point', i.e. of the unique component which has one and only one outgoing arrow and $i$ is $1$ or $2$, depending on the direction of this outgoing arrow (i.e. whether it is $\del_1$ or $\del_2$). For example:
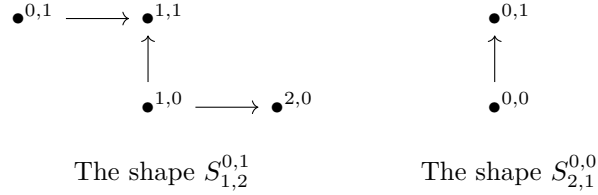
\begin{center}
	
	\begin{tabular}{cccc}
		
		\begin{tikzcd}
			\bullet^{0,1}\ar[r]&\bullet^{1,1}&\\
			&\bullet^{1,0}\ar[u, shift left=2.2]\ar[r]&\bullet^{2,0}
		\end{tikzcd}
		&&&
		\begin{tikzcd}
			\bullet^{0,1}\\
			\bullet^{0,0}\ar[u, shift left=2.2]
		\end{tikzcd}\\
		&&&\\
		The shape $S_{1,2}^{0,1}$&&&The shape $S_{2,1}^{0,0}$
		
	\end{tabular}
\end{center}
 
The numbers $(p,q)\in\Z^2$, $i\in\{1,2\}$ and $r\in\Z_{\geq 1}$ determine the shape uniquely. Let us fix a zigzag $Z$ of shape $S^{p,q}_{i,r}$. The total complex is nonzero only in degree $p+q$ and $p+q+1$ and one may check that the differential is an isomorphism (in fact, it can be described via a triangular matrix with isomorphisms on the diagonal). Therefore, the de Rham cohomology vanishes and the Frölicher spectral sequences have to degenerate. If $i=1$, the row cohomology vanishes completely and therefore $S_2$ is zero on all pages. On the other hand, $_1E_1^{r,s}\neq 0$ exactly for $(r,s)\in \{(p,q),(p+r,q+r-1)\}$ and therefore $_1d^{p,q}_r$, the differential starting in $_1E^{p,q}_r$, has to be the only nonzero differential and $Z^{p+r,q+r-1}=\im {}_1d^{p,q}_r$. Similarly, if $i=2$, the first spectral sequence $S_1$ vanishes on all pages and $Z^{p-r+1,q+r}=\im {}_2d^{p,q}_r$. The following diagram illustrates this for the shape $S_{1,2}^{p,q}$:
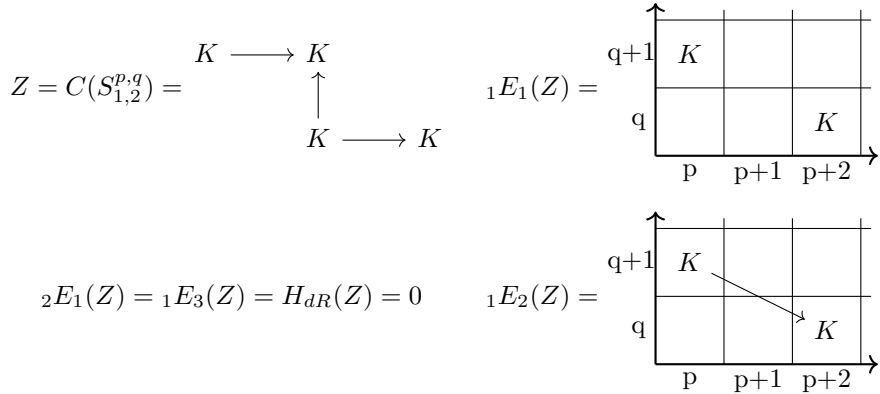
\begin{center}\begin{tabular}{cc}
		$Z=C(S^{p,q}_{1,2})=$\begin{tikzcd}
			K\ar[r]&K&\\
			&K\ar[u]\ar[r]&K
		\end{tikzcd}&$_1E_1(Z)=$\parbox{\textwidth}{
	\begin{tikzpicture}[scale=0.45]
	
	\draw[->, thick] (0,0) -- (6.5,0) node[anchor=west] {};
	\draw (0,2) -- (6.3,2);
	\draw (0,4) -- (6.3,4);
	
	\draw[->, thick] (0,0) -- (0,4.5) node[anchor=south]{};
	\draw (2,0) -- (2,4.3);
	\draw (4,0) -- (4,4.3);
	\draw (6,0) -- (6,4.3);
	
	\draw (1,-.5) node {p};
	\draw (3,-.5) node {p+1};
	\draw (5,-.5) node {p+2};
	
	\draw (-.5,1) node {q};
	\draw (-.7,3) node {q+1};
	
	\node (pq) at (1,3) {$K$};
	\node (p+2q-1) at (5,1) {$K$};
	\end{tikzpicture}}\\
		
$_2E_1(Z)= {}_1E_3(Z)=H_{dR}(Z)=0$&$_1E_2(Z)=$\parbox{\textwidth}{
\begin{tikzpicture}[scale=0.45]

\draw[->, thick] (0,0) -- (6.5,0) node[anchor=west] {};
\draw (0,2) -- (6.3,2);
\draw (0,4) -- (6.3,4);

\draw[->, thick] (0,0) -- (0,4.5) node[anchor=south]{};
\draw (2,0) -- (2,4.3);
\draw (4,0) -- (4,4.3);
\draw (6,0) -- (6,4.3);

\draw (1,-.5) node {p};
\draw (3,-.5) node {p+1};
\draw (5,-.5) node {p+2};

\draw (-.5,1) node {q};
\draw (-.7,3) node {q+1};

\node (pq) at (1,3) {$K$};
\node (p+2q-1) at (5,1) {$K$};

\draw[->] (pq) to (p+2q-1);
\end{tikzpicture}}

	\end{tabular}
\end{center}
\textbf{Odd zigzags:}\\
Let $Z$ be a zigzag of length $2r+1$ with $r\geq 0$ which has most components in total degree $d$ and top left component in bidegree $(a,b)$ and bottom right component $(a',b')$. $Z$ is concentrated at most in total degrees $d,d+1$ or $d,d-1$. In the first case, we denote its shape by $S_d^{a,b'}$ and in the second by $S_d^{a',b}$. E.g.:

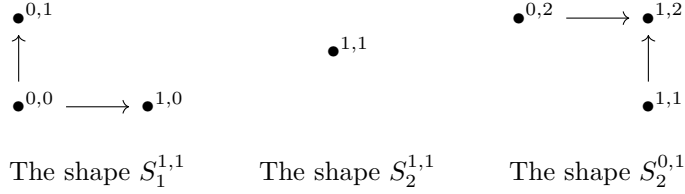
\begin{center}
	\begin{tabular}{ccccc}
		\begin{tikzcd}
			\bullet^{0,1}&\\
			\bullet^{0,0}\ar[u, shift left=2.2]\ar[r]&\bullet^{1,0}
		\end{tikzcd}&&
		\begin{tikzcd}
			\bullet^{1,1}
		\end{tikzcd}&&
		\begin{tikzcd}
			\bullet^{0,2}\ar[r]&\bullet^{1,2}\\
			&\bullet^{1,1}\ar[u, shift left=2.2]
		\end{tikzcd}\\
		&&&&\\
		The shape $S_1^{1,1}$&&The shape $S_2^{1,1}$&& The shape $S_2^{0,1}$
	\end{tabular}
\end{center}
The label $S_d^{p,q}$ determines a unique shape for any triple $(p,q,d)\in\Z$. In fact, a corresponding zigzag has endpoints $(p,d-p)$ and $(d-q,q)$, length $2|p+q-d|+1$ and is concentrated in degrees $d,d-\sgn(p+q-d)$.\\

The Frölicher spectral sequences for $Z$ degenerate on the first page: In fact, since there is an odd number of components but an even number of arrows in each direction, it turns out that the first page of $S_1$ and $S_2$ has only one nonzero component. If $Z$ is concentrated in total degrees $d,d+1$, the nonzero components are $_1E^{a,b}_1$ and $_2E^{a',b'}_1$. Similarly, if $Z$ is concentrated in total degrees $d,d-1$, the nonzero components are $_1E^{a',b'}$ and $_2E^{a,b}$. Hence, in both cases $H^\Cdot_{dR}(Z)=0$ unless $\Cdot=d$, in which case it is one-dimensional. Writing $H:=H_{dR}^d(Z)$, one has $F_1^\Cdot H=H$ for $\Cdot\leq a$ (and zero else) and $F_2^\Cdot H=H$ for $\Cdot\leq b'$ (and zero else) in the first case and $F_1^\Cdot H=H$ only for $\Cdot\leq a'$ and $F_2^\Cdot H=H$ only for $\Cdot \leq b$ in the second case. So for a zigzag of shape $S_d^{p,q}$, the numbers $p$ and $q$ indicate where the Hodge filtrations on $H_{dR}^d(Z)$ jump.\\

The de Rham cohomology can be described more explicitly: Let $\omega^{a,b},...,\omega^{a',b'}$ be generators of the components of $Z$ in degree $d$. If $Z$ is concentrated in degrees $d,d-1$, all $\omega^{r,s}$ are closed and are, up to a constant, cohomologous. I.e., there is a single class which has representatives of $|p+q-d|+1$ pure types. On the other hand, if $Z$ is concentrated in degrees $d,d+1$, a generator for the nontrivial class in $H^d_{dR}(Z)$ is given by a linear combination of the $\omega^{r,s}$ in which all $\omega^{r,s}$ have nonzero coefficient, i.e. one has a class of (generally) very `non-pure type'.\\

We illustrate the above discussion for a zigzag $Z$ of shape $S_{p+q}^{p+1,q+1}$:
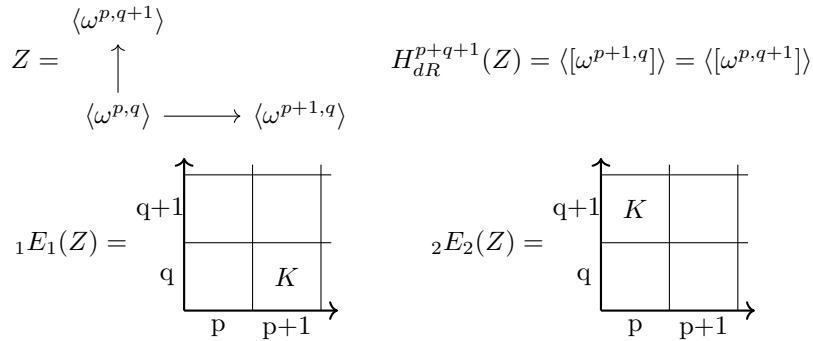
\begin{center}\begin{tabular}{cc}
		$Z=$\begin{tikzcd}
			\langle \omega^{p,q+1}\rangle&\\
			\langle\omega^{p,q}\rangle \ar[u]\ar[r]&\langle\omega^{p+1,q}\rangle
		\end{tikzcd}& $H_{dR}^{p+q+1}(Z)=\langle[\omega^{p+1,q}]\rangle=\langle[\omega^{p,q+1}]\rangle$\\
		
		$_1E_1(Z)=$\parbox{3cm}{
			\begin{tikzpicture}[scale=0.45]
			
			\draw[->, thick] (0,0) -- (4.5,0) node[anchor=west] {};
			\draw (0,2) -- (4.3,2);
			\draw (0,4) -- (4.3,4);
			
			\draw[->, thick] (0,0) -- (0,4.5) node[anchor=south]{};
			\draw (2,0) -- (2,4.3);
			\draw (4,0) -- (4,4.3);
			
			\draw (1,-.5) node {p};
			\draw (3,-.5) node {p+1};
			
			\draw (-.5,1) node {q};
			\draw (-.7,3) node {q+1};
			
			\node (pq) at (3,1) {$K$};
			\end{tikzpicture}}
		&
		$_2E_2(Z)=$\parbox{3cm}{
			\begin{tikzpicture}[scale=0.45]
			
			\draw[->, thick] (0,0) -- (4.5,0) node[anchor=west] {};
			\draw (0,2) -- (4.3,2);
			\draw (0,4) -- (4.3,4);
			
			\draw[->, thick] (0,0) -- (0,4.5) node[anchor=south]{};
			\draw (2,0) -- (2,4.3);
			\draw (4,0) -- (4,4.3);
			
			\draw (1,-.5) node {p};
			\draw (3,-.5) node {p+1};
			
			\draw (-.5,1) node {q};
			\draw (-.7,3) node {q+1};
			
			\node (pq) at (1,3) {$K$};
			\end{tikzpicture}}
		
	\end{tabular}
\end{center}

\textbf{Squares:}\\
A square shape is defined by the position of any corner. We choose the top right one and define $S^{p,q}:=\{(p,q),(p-1,q),(p,q-1),(p-1,q-1)\}$. For any square $A$ of shape $S^{p,q}$, one sees that row and column cohomology vanish in every degree. Therefore, all higher pages of the Frölicher spectral sequences and the de Rham cohomology have to vanish as well. However, one has 
\[
\im(\del_1\circ\del_2)^{p,q}=A^{p,q}.
\]
Note also that $\del_1\circ\del_2$ vanishes on all zigzags.\\

Since `everything' is compatible with direct sums, the above discussion of the individual indecomposable complexes yields several results for general $A$:

\begin{prop}\label{multiplicities and cohomology}
	Let $A$ be a bounded double complex over $K$.
	\begin{enumerate}
		\item \textbf{Even zigzags:}
		There is an equality
		\[
		\mult_{S_{i,r}^{p,q}}(A)=\dim\im{}_id_r^{p,q}.
		\]
		
		\item \textbf{Odd zigzags:} Given $\varphi:\bigoplus T_i\overset{\sim}{\rightarrow} A$ a decomposition into elementary complexes with distinct shapes, denote $H_{\varphi,d}^{p,q}:=H_{dR}^d(\varphi T_i)\subseteq H^{d}_{dR}(A)$ if $S(T_i)=S^{p,q}_d$. These spaces split the filtrations $F_1$ and $F_2$, i.e.
		\[
		F_1^pH_{dR}^d(A)=\bigoplus_{r\geq p}H^{r,s}_{\varphi, d},\qquad F_2^qH^{d}_{dR}(A)=\bigoplus_{s\geq q}H^{r,s}_{\varphi,d}.
		\]
		In particular, 
		\[
		\mult_{S_d^{p,q}}(A)=\dim \gr_{F_1}^p\gr_{F_2}^qH^d_{dR}(A).
		\]
		\item \textbf{Squares:} There is an equality 
		\[
		\mult_{S^{p,q}}(A)=\dim(\im \del_1\circ\del_2)\cap A^{p,q}.
		\]
	\end{enumerate}
\end{prop}

It was shown in \cite{deligne_real_1975} that a double complex satisfies the $\del_1\del_2$-lemma iff it has degenerate Frölicher spectral sequences and the $k$-th total cohomology has a pure Hodge structure of weight $k$ iff it is a direct sum of squares and zigzags of length $1$. The following corollary, together with Theorem \ref{decomposition of double complexes} is a generalisation of this last equivalence.

\begin{cor}\label{generalised DGMS} Let $A$ be a bounded double complex over a field $K$.
\begin{itemize}
		\item The Frölicher spectral sequences degenerate at stage $r$ if and only if only shapes of even zigzags of length strictly less than $2r$ have nonzero multiplicity.
		
		\item  Fix some degree $d\in\Z$ and set $H:=H_{dR}^d(A)$ and for any $p,q\in\Z$ set $H^{p,q}:=F_1^pH\cap F_2^qH$. The space $H$ carries a pure Hodge structure of weight $k$ (i.e., $H=\bigoplus_{p+q=k} H^{p,q}$ ) if and only if all odd length zigzag shapes of the form $S^{p,q}_d$ which have nonzero multiplicity in $A$ satisfy $p+q=k$. Further, the spaces $H^{p,q}$ admit the following description:
		If $p+q\geq d$:
		\[
		H^{p,q}=\left\{\parbox{6.6cm}{classes that admit a representative ${\omega\in A^{r,s}}$ for all $(r,s)\in\{(d-q,q),...,(p,d-p)\}$}\right\}.
		\]
		If $p+q\leq d$:
		\[
		H^{p,q}=\left\{\parbox{6cm}{classes that admit a representative ${\omega=\sum_{j=p}^{d-q}\omega_{j,d-j}}$ with $\omega_{r,s}\in A^{r,s}$.}\right\}.
		\]
	
	\end{itemize}
	If the involved quantities are finite, the first point is equivalent to the equality $\sum_{p,q\in\Z} \dim E_r^{p,q}=\sum_{k\in\Z} b_k$ and the second to $b_d=\sum_{p+q=k}b_d^{p,q}$.
\end{cor}

Proposition \ref{multiplicities and cohomology}.2. states in particular that the Betti-numbers (i.e. the dimensions of de Rham cohomology) count the odd zigzags contributing to a certain total degree. Similarly, the dimensions of Dolbeault cohomology count the zigzags starting or ending in a certain bidegree. The dimensions of Bott-Chern (resp. Aeppli) cohomology count `corners', i.e. the zigzags meeting a certain bidegree with possibly incoming (resp. outgoing) arrows. The following is a more technical formulation of these sentences:

\begin{lem} \label{Dolbeault, Bott Chern and Aeppli via zigzags}
Let $A$ be a bounded double complex with a decomposition into elementary complexes with pairwise distinct support $\varphi:\bigoplus A_i\overset{\sim}{\rightarrow} A$.
\begin{itemize}
	\item 
	For every $(p,q)\in\Z^2$, the maps induced by $\varphi$
	\[
	\bigoplus_{\substack{A_i\text{ zigzag}\\(p,q)\in S(A_i)\\(p-1,q),(p+1,q)\not\in S(A_i)}}H_{\del_1}^{p,q}(A_i)\longrightarrow H^{p,q}_{\del_1}(A)
	\]
	\[
	\bigoplus_{\substack{A_i\text{ zigzag}\\(p,q)\in S(A_i)\\(p,q-1),(p,q+1)\not\in S(A_i)}}H_{\del_2}^{p,q}(A_i)\longrightarrow H^{p,q}_{\del_2}(A)	
	\]
	are isomorphisms.
	\item 
	For every $(p,q)\in\Z^2$, the maps induced by $\varphi$
	\[
	\bigoplus_{\substack{A_i\text{ zigzag}\\(p,q)\in S(A_i)\\(p+1,q),(p,q+1)\not\in S(A_i)}}H_{BC}^{p,q}(A_i)\longrightarrow H^{p,q}_{BC}(A)
	\]
	\[
	\bigoplus_{\substack{A_i\text{ zigzag}\\(p,q)\in S(A_i)\\(p-1,q),(p,q-1)\not\in S(A_i)}}H_{A}^{p,q}(A_i)\longrightarrow H^{p,q}_A(A)	
	\]
	are isomorphisms.
\end{itemize}
\end{lem}
In particular, one sees that both $H_{BC}^{p,q}(A)$ and $H_{A}^{p,q}(A)$ are finite dimensional for all $(p,q)\in\Z^2$ if and only if this is true for $H_{\del_1}^{p,q}(A)$ and $H_{\del_2}^{p,q}(A)$.\\

We now illustrate how the theory developed so far can be used to turn certain heuristics (see e.g. \cite{angella_bott-chern_2015} and \cite{angella_hodge_2018}) previously used into actual proofs, by reproving, and slightly generalising, the main result of \cite{angella_$partialoverlinepartial$_2013} via `counting zigzags':\\

For any double complex $A$ such that all involved quantities are finite the non-$\del_1\del_2$-degrees are defined as
\[
\Delta^k(A):=\sum_{p+q=k}(\dim H^{p,q}_{BC}(A)+\dim H^{p,q}_A(A))-2\dim H^k_{dR}(A).
\]
These were studied in \cite{angella_$partialoverlinepartial$_2013} for the double complex of forms on a compact complex manifold.
By definition, the $\Delta^k(A)$ are additive under direct sums. One also verifies that they vanish on squares and dots.\\
If $Z$ is a zigzag of shape $S_d^{p,q}$ with $p+q\neq d$ one has 
\[
\Delta^k(Z)=\begin{cases}
|p+q-d|-1&\text{if }k=d\\
|p+q-d|&\text{if }k=d-\sgn(p+q-d)\\
0&\text{else.}
\end{cases}
\]

Similarly, for a zigzag $Z$ of even length with shape $S_{i,l}^{p,q}$ one may verify that \[
\Delta^k(Z)=\begin{cases}
l&\text{if }k=p+q,p+q+1\\
0&\text{else.}
\end{cases}
\]

In particular, one obtains a generalisation to arbitrary bounded double complexes over field of the main result of \cite{angella_$partialoverlinepartial$_2013}, also explaining the name non-$\del_1\del_2$-degrees:

\begin{thm}\label{non-deldelbar-degrees}
For any bounded double complex $A$ over a field with $H_{BC}^{p,q}$ and $H_A^{p,q}(A)$ finite for all $p,q\in\Z$,  one has 
\[
\Delta^k(A)\geq 0\qquad\text{for all }k\in\Z
\]
and equality holds if and only if $A$ satisfies the $\del_1\del_2$-lemma.
\end{thm}

Now we turn to a new notion of quasi-isomorphism: We saw that all the cohomological information is encoded in the zigzags and all the information about the zigzags is encoded in the Frölicher spectral sequences and the bifiltered de Rham cohomology. This motivates the following refined version of quasi-isomorphism (cf. \cite{cirici_model_2019} and \cite{neisendorfer_dolbeault_1978} for different but related notions):

\begin{definition}\label{def: quiso}
Let $r\in\Z_{\geq 0}\cup\{\infty\}$. A morphism $f:A\longrightarrow B$ of bounded double complexes is called an \textbf{$E_r$-isomorphism} if it induces an isomorphism on the $r$-th page of both Frölicher spectral sequences (where the $0$-th page is defined to be the complex itself).
\end{definition}

So an $E_0$-isomorphism is just an ordinary isomorphism and we will usually reserve the name $E_r$-isomorphism for the cases $r\geq 1$. For example, a morphism is an $E_1$-isomorphism iff it induces a morphism in Dolbeault cohomology and an $E_\infty$-isomorphism iff it induces a isomorphism in de Rham cohomology which is strictly compatible with both filtrations (i.e. it induces isomorphisms $F_i^p\rightarrow F_i^p$). Any $E_r$-isomorphism is also an $E_{r'}$-isomorphism for any $r'>r$.\\

Our next goal are two results on $E_r$-isomorphisms. One is a characterisation in of the notion of $E_r$-isomorphism in terms of indecomposable summands and the other is an easy-to-check criterion for a functor to send $E_r$-quasi-isomorphisms to isomorphisms.

\begin{prop}
\label{characterisation of E_r-isomorphisms}Let $A$ and $B$ be bounded double complexes over $K$.
\begin{itemize}
	\item[i)] A morphism $f:A\longrightarrow B$ is an isomorphism (resp. $E_r$-isomorphism for $r\geq 1$) if and only if for any decomposition into elementary complexes with pairwise distinct support $A\cong \bigoplus C(S)^{\oplus\mult_S(A)}$, $B\cong \bigoplus C(S)^{\oplus\mult_S(B)}$, the induced map $f_{S,S}:C(S)^{\oplus\mult_S(A)}\longrightarrow C(S)^{\oplus\mult_{S'}(B)}$ is an isomorphism for all $S$ (resp. for all $S$ which are zigzag shapes of odd length or of even length $\geq 2r$). 
	
	\item[ii)] There exists an isomorphism (resp. $E_r$-isomorphism for $r\geq 1)$ between $A$ and $B$ if and only if $\mult_S(A)=\mult_S(B)$ for all shapes (resp. all zigzag shapes $S$ of odd length or of even length $\geq 2r$).
\end{itemize}	
In particular, the relation
\[
A\simeq_r B:\Leftrightarrow \text{there exists an }E_r\text{-isomorphism }A\longrightarrow B
\]
is an equivalence relation. 
\end{prop}

For example, two complexes are $E_1$-isomorphic iff all zigzags occur with the same multiplicity and $E_\infty$-isomorphic iff all odd zigzags occur with the same multiplicity.

\begin{prop}\label{linear functors and quasi-isomorphisms}
	An $E_r$-isomorphism $f:A\longrightarrow B$ ($r\geq 1$) induces an isomorphism after applying any linear functor from double complexes to vector spaces which sends squares and even-length zigzags of length $<2r$ to the zero vector space. 
\end{prop}

In particular, one immediately recovers the following fact (see also \cite{angella_cohomologies_2013-1}, \cite{angella_bott-chern_2017}, for previous special cases):

\begin{cor}\label{E_1 q.i. implies BC q.i.}
	An $E_1$-isomorphism induces an isomorphism in Bott-Chern and Aeppli cohomology.
\end{cor}

The proof of these results will rest on the fact that, after some simplification steps, one can bring $E_r$-isomorphisms to `triangular shape'. To make this precise, we need a few technical preparations:\\

For any shape $S$, denote by $\deg S:=\min\{p+q\mid (p,q)\in S\}$ the degree in which an elementary complex of this shape is generated. Let $\mathcal{S}$ be a finite set of shapes. We are going to define a total ordering on $\mathcal{S}$ as follows:\\

\textbf{Order by total degree:} Set $S\leq S'$ if $\ord S< \ord S'$.\\

\textbf{Order within a total degree:} Let $\mathcal{S}_d\subseteq \cS$ the subset of all shapes of degree $d$. Set $p_d:=\min\{p\mid (p,d-p)\in S \text{ for some } S\in\cS\}$ and similarly $q_d:=\max\{q\mid (d-q,q)\in S\text{ for some } S\in\cS\}$.\\

Since $\cS_d$ is finite, a total order on $\cS_d$ is the same as writing down all elements in a list, with elements appearing later being declared greater than previous ones. We compile such a list as follows:

\begin{enumerate}
	\item For each $n$, starting with $n=0$, add, in arbitrary order, all the shapes with $2n+1$ elements in $\cS_d$, $n+1$ of which have total degree $n$. (I.e. first all dots, then all length-three-shapes contributing to total degree $d$, etc.)
	\item For $p$, starting at $p_d$ and going up to $d-q_d$: Add all even length shapes in $\cS_d$ containing $(p,d-p)$ and $(p,d-p+1)$ but not $(p-1,d-p+1)$, ordered by increasing length. (I.e. first $S_{2,1}^{{p_d,d-p_d}}$, then $S_{2,2}^{p_d+1,d-p_d-1}$ etc. After that $S_{2,1}^{p_d+1,d-p_d-1}$ etc. and so on.)
	\item Analogously, for $q$ starting at $q_d$ and going down to $d-p_d$: Add all even length shapes in $\cS_d$ containing $(d-q,q)$ and $(d-q+1,q)$, ordered by increasing length.
	\item For each $n$, starting with $n$ as large as possible and decreasing, add, in arbitrary order, all the shapes with $2n+1$ elements in $\cS_d$, $n+1$ of which have total degree $n+1$.
	\item Add all square shapes in $\cS_d$, in arbitrary order.
\end{enumerate}
One then verifies the following tedious but elementary lemma:

\begin{lem}\label{Partial ordering on zigzags}
	For two elementary complexes $Z,Z'$ of shapes $S,S'\in \cS$, if $S\leq S'$, then 
	\[
	\Hom(Z,Z')=0
	\]
\end{lem}

Using this, let us prove the above results.

\begin{proof}[Proof of Prop. \ref{characterisation of E_r-isomorphisms}]
	Fix two decompositions $A\cong A^{<2r}\oplus A^{zig}$, $B\cong B^{<2r}\oplus B^{zig}$ in which the first summand consists of squares and even length zigzags of length $<2r$ (it is trivial if $r=0$) and the second one of odd zigzags and even ones of length $\geq 2r$. Denote by $\cS$ the set of all shapes occuring in $A^{zig}$ and $B^{zig}$ (a posteriori, it would be sufficient to consider just one of them).\\
	
	For the `if'-part of $i)$, note that, since we only care for the behaviour after applying $E_r$, we may assume that $A=A^{zig}$ and $B=B^{zig}$. Now, ordering the shapes in $\cS$ as before and decomposing $A\cong\bigoplus_{S\in\cS}C(S)^{\oplus\mult_S(A)}$ and $B\cong\bigoplus_{S\in\cS}C(S)^{\oplus\mult_S(B)}$, the resulting `matrix' $(f_{S,S'})_{S,S'\in\cS}$ of induced maps $f_{S,S'}:C(S)^{\oplus\mult_S(A)}\longrightarrow C(S')^{\oplus\mult_{S'}(B)}$ has triangular shape by Lemma \ref{Partial ordering on zigzags} and the diagonal entries $f_{S,S}$ are invertible by assumption. So, $f$ induces an isomorphism $A^{zig}\cong B^{zig}$, and therefore necessarily one after applying $E_r$.\\
	
	For the `only if' part of $i)$, one uses that a map between two elementary complexes of the same shape is an isomorphism iff it is an isomorphism in some bidegree and that the map $f_{S,S}$ may, at least in a certain bidegree, canonically factored through the spaces occuring in Proposition \ref{multiplicities and cohomology}. For example, let $A'$ and $B'$ be elementary complexes of shape $S^{p,q}$ in some decomposition of $A$ and $B$. The inclusion $A'\subseteq A$ induces an isomorphism ${A'}^{p,q}\cong \im \del_1\del_2\cap A^{p,q}$ and similarly the projection $B\rightarrow B'$ an isomorphism $\im\del_1\del_2\cap B^{p,q}\cong {B'}^{p,q}$. In particular,  whenever $f$ induces an iso on $\im\del_1\del_2$ in degree $(p,q)$, it induces an iso $A'\cong B'$ and vice versa. The cases of even and odd zigzags may be treated similarly.\\
	
	The `only if' part of $ii)$ is a consequence of the first since any $E_r$-isomorphism induces an isomorphism on the spaces whose dimensions encode the relevant multiplicites. On the other hand, given decompositions $\varphi:A\cong\bigoplus C(S)^{\oplus\mult_S(A)}$, $\psi: B\cong \bigoplus C(S)^{\oplus\mult_S(B)}$ and $\mult_S(A)=\mult_S(B)$ for all $S$ in a certain set of shapes $\mathcal{S}$, one may construct a map inducing an isomorphism on all elementary components with shapes in $\mathcal{S}$ as $\varphi$ followed by the projection to $\bigoplus_{S\in\mathcal{S}}C(S)^{\oplus\mult_S(A)}$, followed by the inclusion and $\psi^{-1}$.
\end{proof}

\begin{proof}[Proof of Prop. \ref{linear functors and quasi-isomorphisms}]
Let $F$ be a functor as in the statement. Since it is linear, it automatically commutes with direct sums. Decompose source and target of $f$ into $A=A^{sq}\oplus A^{zig}$, $B=B^{sq}\oplus B^{zig}$. Since $f$ is an $E_1$-quasi-isomorphism, it sends $A^{sq}$ to $B^{sq}$ and therefore it makes sense to consider the reduced map $\bar{f}:A^{zig}\cong A/A^{sq}\longrightarrow B/B^{sq}\cong B^{zig}$. As $F$ sends squares to $0$, $F(f)$ may be identified with $F(\bar{f})$ and we may assume $A,B$ do not contain any squares.\\

Decomposing $A=A^{zig}\bigoplus C(S)^{\oplus\mult_S(A)}$ and $B=B^{zig}=\bigoplus C(S)^{\oplus\mult_S(B)}$ into zigzags, the induced map $F(f)$ can be written as a `matrix' with rows and columns indexed by zigzag-shapes. Again using the previous ordering, this matrix has triangular shape with invertible elements on the diagonal as a result of Proposition \ref{characterisation of E_r-isomorphisms} and Lemma \ref{Partial ordering on zigzags}.
\end{proof}

\section{Tensor product and Grothendieck rings}\label{sec: Grothendieck rings}
We are going to investigate the behaviour of the above decomposition under tensor product and compute the Grothendieck rings of several categories of double complexes.\\

For any $r\geq 0$ we consider the following categories
\begin{center}
	\begin{tabular}{l p{7cm}}
		$\operatorname{DC}_K^{E_r-fin, b}$& bounded double complexes over $K$ s.t. the $E_r$-page of both Frölicher spectral sequences is finite dimensional, localised at $E_r$-isomorphisms.\\
		&\\	
		$\operatorname{DC}_K^{\del_1\del_2-fin, b}$& bounded double complex satisfying the $\del_1\del_2$-lemma s.t. the $E_\infty$ page of both Frölicher spectral sequences is finite dimensional, localised at $E_\infty$-isomorphisms.		
	\end{tabular}
\end{center}

Note that $\operatorname{DC}_K^{E_0-fin, b}$ is just the subcategory of $\DC_K^{b}$ consisting of finite dimensional complexes and that by Corollary \ref{generalised DGMS} we could replace $E_\infty$  by $E_r$ for any $r\geq 1$ in the definition of $\DC_K^{\del_1\del_2-fin, b}$ since the Frölicher spectral sequences degenerate.\\

The following lemma is a consequence of the Künneth formula and the compatibility of (co)homology with direct sums.
\begin{lem}\label{spectral sequences of sum and tensor product}
	Let $A,B$ be double complexes. For every $r\geq 0$, there are functorial isomorphisms
	\begin{align*}
	_iE_r(A\oplus B)&\cong{} _iE_r(A)\oplus {} _iE_r(B)\\
	_iE_r(A\otimes B)&\cong{} _iE_r(A)\otimes {} _iE_r(B).
	\end{align*}
\end{lem}

As a consequence of this lemma, direct sum and tensor product are well-defined on the categories defined above and we can define the following Grothendieck rings (i.e. formal sums of isomorphism classes modulo the relation $[A]+[B]=[A\oplus B]$ and with multiplication induced by tensor product):
\begin{align*}
\cR_r&:=K_0(\DC_K^{E_r-fin,b})\\
\cR_{\del_1\del_2}&:= K_0(\DC_K^{\del_1\del_2-fin, b})
\end{align*}

Given a double complex $A$ with suitable finiteness conditions, write $[A]$ for its class in one of these rings. Abusing notation slightly, given a shape $S$ we write $[S]$ for the class of some elementary complex of rank $1$ with shape $S$. By Theorem \ref{decomposition of double complexes}, equations of the form
\[\tag{$\ast$}
[A]=\sum_{\substack{S \text{ shape}\\ [S]\neq0}}\mult_S(A)[S]
\]
hold.\\

So, as an abelian group, $\cR_0$ (resp. $\cR_\infty$, resp. $\cR_r$ for any $r\in \Z_{>0}$) is free with basis given by all shapes (resp. all odd zigzag shapes, resp. all odd zigzag shapes and all even zigzag shapes of length $\geq 2r$). We now describe the multiplicative structure.

\begin{prop}\label{multiplication rules for shapes}
	For a square shape $S^{p,q}$ and any other shape $S$, there is an equality in $\cR_0$:
	\[
	[S^{p,q}]\cdot[S]=\sum_{(r,s)\in S}[S^{p+r,q+s}].
	\]
	In particular, the subgroup generated by square shapes is an ideal $I_{Sq}$ in $\cR_0$. There are equalities in $\cR_1\cong\cR_0/I_{Sq}$:
	
	\begin{align*}
	\Big[S^{p,q}_d\Big]\cdot \Big[S^{p',q'}_{d'}\Big]&=\Big[S^{p+p',q+q'}_{d+d'}\Big]\\
	\Big[S^{p,q}_{1,l}\Big]\cdot\Big[S^{p',q'}_{1,l'}\Big]&=\Big[S_{1,\min(l,l')}^{p+p',q+q'}\Big]+\Big[S_{1,\min(l,l')}^{p+p'+\max(l,l'),q+q'-\max(l,l')+1}\Big]\\
	\Big[S^{p,q}_{2,l}\Big]\cdot\Big[S^{p',q'}_{2,l'}\Big]&=\Big[S_{2,\min(l,l')}^{p+p',q+q'}\Big]+\Big[S_{2,\min(l,l')}^{p+p'-\max(l,l')+1,q+q'+\max(l,l')}\Big]\\
	\Big[S^{p,q}_{1,l}\Big]\cdot\Big[S^{p',q'}_{2,l'}\Big]&=0\\
	\Big[S_d^{p,q}\Big]\cdot\Big[S_{1,l}^{p',q'}\Big]&=\Big[S_{1,l}^{p+p',q'+d-p}\Big]\\
	\Big[S_d^{p,q}\Big]\cdot\Big[S_{2,l}^{p',q'}\Big]&=\Big[S_{2,l}^{p'+d-q,q+q'}\Big]
	\end{align*}
\end{prop}
\begin{proof}
	To see the equation for squares, let $Z$ be an indecomposable complex with shape $S^{p,q}$ and $Z'$ an indecomposable complex with shape $S$. Choose a basis element $s\in Z^{p-1,q-1}$. In particular, $\del_1\del_2 s\neq 0$. Given a basis element $\alpha^{r,s}$ of any nonzero component ${Z'}^{r,s}$, the element
	\[
	\del_1\del_2(s\otimes\alpha^{r,s})=\del_1\del_2s\otimes\alpha^{r,s}+s\otimes\del_1\del_2\alpha^{r,s}\in Z^{p,q}\otimes {Z'}^{r,s}\oplus Z^{p-1,q-1}\otimes {Z'}^{r+1,s+1}
	\]
	is not zero, so one obtains
	\[
	\mult_{S^{p+r,q+s}}(Z\otimes Z')\geq 1
	\]
	whenever ${Z'}^{r,s}\neq 0$ and so one has 
	\begin{align*}
	\dim(Z\otimes Z')&\geq 4\cdot\sum_{(r,s)\in S}\mult_{S^{p+r,q+s}}(Z\otimes Z')\\
	&\geq 4\cdot\dim Z'\\
	&\geq \dim (Z\otimes Z')
	\end{align*} and hence equality, which implies the formula.\\
	
	The other equations all follow from a consideration of the Frölicher spectral sequences for two indecomposable complexes with the given shapes and using Proposition \ref{multiplicities and cohomology}. In each case there are only very few (i.e. $\leq 2$) nonzero entries on each page. We only do this for the most terrible looking formula, the others follow similarly:\\
	
	Let $l\leq l'$ and $Z,Z'$ elementary double complexes of rank one with shapes $S_{1,l}^{p,q}$ and $S_{1,l'}^{p',q'}$. Then, $_2E_r(Z)={}_2E_r(Z')=0$ for all $r\geq 1$. Therefore, $_2E_r(Z\otimes Z')=0$ for all $r\geq 1$ and so 
	\[
	\mult_{S^{a,b}_{2,d}}(Z\otimes Z')=0
	\] for all $a,b\in\Z,d\in\Z_{> 0}$ and
	\[
	\mult_{S}(Z\otimes Z')=0
	\]
	for all odd length zigzag shapes $S$, since the total cohomology has to vanish.\\
	
	Considering the other spectral sequence, one has 
	$_1E_r(Z)=0$ for $r>l$ and if $r\leq l$ it is nonzero only in bidegrees $(p,q)$ and $(p+l,q-l+1)$, where it has dimension $1$. Similarly, $_1E_r(Z')=0$ for $r>l'$ and if $r\leq l'$, it is nonzero only in bidegrees $(p,q)$ and $(p+l',q-l'+1)$, where it has dimension $1$.\\
	
	In summary, $_1E_r(Z\otimes Z')=0$ for all $r> l=\min (l,l')$ and nonzero in bidegrees $(p+p',q+q'),(p+p'+l,q+q'-l+1),(p+p'+l',q+q'-l'+1),(p+p'+l+l',q+q'-l-l'+2)$, so the two necesary nonzero differentials of bidegree $(l,-l+1)$ on page $_1E_l$ have to start at bidegrees $(p+p',q+q')$ and $(p+p'+l',q+q'-l'+1)$ and all other differentials vanish.
\end{proof}
The last equalities can be memorised by the following rules, alluding to the parity of the zigzag length:
\begin{center}
	even$\cdot$even = even\qquad odd$\cdot$odd=odd\qquad odd$\cdot$even=even
\end{center}

These multiplication rules allow several immediate conclusions:
\begin{itemize}
	\item The maps of abelian groups $\cR_\infty\longrightarrow \cR_r$ induced for any $r\in\Z_{>0}$ by the equations $(\ast)$ are maps of rings, i.e., $\cR_r$ is a $\cR_\infty$-algebra.
	\item For $r\in\Z_{>0}$, the rings $R_{r}$ are not finitely generated as $\Z$-algebras. In fact, for any hypothetical finite set of generators there would be a natural number $l_0$ s.t. the length of all even length zigzag occuring as summands in the generators is bounded by $l_0$. By the multiplication rules, this would also be true for all sums of products of these hypothetical generators.
\end{itemize}

We will now describe $\cR_r$ for all $r\geq 1$ using generators and relations. We invite the reader to draw pictures of all generators and relations to see what is going on.

\begin{thm}\label{Grothendieck ring of dc}
	Let $\mathcal{P}:=\Z[R^{\pm 1},U^{\pm 1},L^{\pm 1},\{X_l\}_{l\in\Z_{>0}}\{Y_l\}_{l\in\Z_{>0}}]$ be the polynomial ring in infinitely many generators with inverses for three variables. The map
	\begin{align*}
	\Phi: \mathcal{P}\longrightarrow&\cR_1\\
				&\\
	R\longmapsto & [S_1^{0,1}]\\
	U\longmapsto & [S_1^{1,0}]\\
	L\longmapsto & [S_1^{1,1}]\\
	X_l\longmapsto & [S_{1,l}^{0,0}]\\
	Y_l\longmapsto & [S_{2,l}^{0,0}]\\
	\end{align*}
	is surjective and the kernel is described by enforcing the relations
	\begin{align*}
		X_l\cdot Y_{l'}&=0\\
		R\cdot X_l&=L\cdot X_l\\
		 R\cdot Y_l&= L^{-1}\cdot Y_l\\
		U\cdot X_l&=L^{-1}\cdot X_l\\
		 U\cdot Y_l&= L\cdot Y_l\\
		X_{l'}\cdot X_{l}&=X_l+R^{l'}\cdot U^{-l'+1}\cdot X_l\\
		 Y_{l'}\cdot Y_{l}&= Y_l + R^{-l'+1}\cdot U^{l'}\cdot Y_l
	\end{align*}
	for all $l'\geq l$ in $\Z_{> 0}$.\\
	
	Via the `same' map, $\cR_\infty$ and $\cR_{\del_1\del_2}$ are identified with $\Z[R^{\pm 1},U^{\pm 1}, L^{\pm 1}]$ and $\Z[R^{\pm 1}, U^{\pm 1}]$.
\end{thm}

\begin{proof}
That the map $\Phi$ is well-defined and the indicated relations have to hold in the image follows from Proposition \ref{multiplication rules for shapes}.\\

Let us show that the induced map $\overline\Phi$ from $\Phi$ modulo the above relations is indeed an isomorphism: Given a polynomial $P\in\mathcal{P}$ we can, using only the given relations, always arrange it to a sum
\[
P=P_\infty+P_X+P_Y
\]
where
\begin{align*}
P_\infty&\in\Z[R^{\pm 1}, U^{\pm 1}, L^{\pm 1}]\\
P_X&\in \bigoplus_{l>0}\Z[R^{\pm 1}, U^{\pm 1}]X_l\\
P_Y&\in\bigoplus_{l>0}\Z[R^{\pm 1}, U^{\pm 1}]Y_l
\end{align*}

Then $\Phi(P_\infty)$ consists only of odd zigzags, $\Phi(P_X)$ (resp. $\Phi(P_Y)$) only of even zigzags contributing to some page in $S_1$ (resp. $S_2$). In particular, the given representation is necessarily unique. One concludes using the following two observations:
\begin{itemize}
	\item Multiplication by $\Phi(R^{\pm 1})$ and $\Phi(U^{\pm 1})$ acts as horizontal and vertical shifts on all other zigzag shapes.
	\item the powers of $\Phi(L^{\pm 1})$  (resp. the images of $X_l$ and $Y_l$) are precisely a set of representatives for the equivalence classes of odd zigzags (resp. even zigzags contributing to $S_1$ or $S_2$) up to shifts.
\end{itemize}

\end{proof}

In many `real world' cases, all double complexes one might be interested in are concentrated in the first quadrant, i.e., $A^{p,q}=0$ whenever $p<0$ or $q<0$. We denote by $\DC_K^{E_r-fin, b, +}$ and $\DC_K^{\del_1\del_2, b, +}$ the analogues of the categories at the beginning of this sections where we assume in addition that all double complexes are (at least from the $E_r$-page on) concentrated in the first quadrant. The corresponding Grothendieck rings will be denoted $\mathcal{R}_r^+$ and $\cR_{\del_1\del_2}^+$. The proof of the following corollary is very similar to that of the previous, so we omit it:

\begin{cor}\label{Grothendieck ring of upper quadrant dc}
	Let $\mathcal{P}^+:=\Z[R, U, L, \widetilde{L}, \{\widetilde{X}_l\}_{l\in\Z>0},\{\widetilde{Y}_l\}_{l\in\Z_{>0}}]$ be a polynomial ring in infinitely many generators. The map
	\begin{align*}
	\Phi^+:\mathcal{P}^+\longrightarrow&\cR_1^+\\
	&\\
	R\longmapsto & [S_1^{0,1}]\\
	U\longmapsto & [S_1^{1,0}]\\
	L\longmapsto & [S_1^{1,1}]\\
	\widetilde{L}\longmapsto & [S^{0,0}_1]\\
	\widetilde{X}_l\longmapsto & [S_{1,l}^{0,l-1}]\\
	\widetilde{Y}_l\longmapsto & [S_{2,l}^{l-1,0}]\\
	\end{align*}
	is surjective and the kernel is described by enforcing the relations
	\begin{align*}
	R\cdot U&=L\cdot \widetilde{L}\\
	\widetilde{X}_l\cdot \widetilde{Y}_{l'}&=0\\
	R\cdot \widetilde{X}_l&=L\cdot \widetilde{X}_l\\
	R\cdot \widetilde{Y}_l&= \widetilde{L}\cdot \widetilde{Y}_l\\
	U\cdot \widetilde{X}_l&=\widetilde{L}\cdot \widetilde{X}_l\\
	U\cdot \widetilde{Y}_l&= L\cdot \widetilde{Y}_l\\
	X_{l'}\cdot \widetilde{X}_{l}&=U^{l'-1}\widetilde{X}_l+R^{l'}\cdot \widetilde{X}_l\\
	Y_{l'}\cdot \widetilde{Y}_{l}&= R^{l'-1}\widetilde{Y}_l + U^{l'}\cdot \widetilde{Y}_l
	\end{align*}
	for all $l'\geq l$ in $\Z_{> 0}$.\\
	
	Via the `same' map, $\cR_\infty^+$ and $\cR_{\del_1\del_2}^+$ are identified with $\Z[R,U, L,\widetilde{L}]/(R\cdot U-L\cdot\widetilde{L})$ and $\Z[R, U]$.\\
	
	The inclusion $\cR_1^+\hookrightarrow \cR_1$ is induced by the map
	\begin{align*}
	\mathcal{P}^+\longrightarrow&\mathcal P
	\end{align*}
	defined by 
	\[
	R\mapsto R,\quad U\mapsto U,\quad L\mapsto L,\quad \widetilde{L}\mapsto RUL^{-1},\quad \widetilde{X}_l\mapsto U^{l-1}X_l,\quad \widetilde{Y}_l\mapsto R^{l-1}Y_l.
	\]
\end{cor}

\section{Complex manifolds}\label{sec: complex manifolds}

In this section we apply the theory developed so far to the double complex \[\cA_X=(\cA_X^{\Cdot,\Cdot},\del,\delbar)\] of $\C$-valued differential forms on a complex manifold $X$ of pure dimension $n$.\\

First, we reinterpret several known results on $\cA_X$ in terms of the decomposition into indecomposables. Some of them are also summarised briefly in \cite{angella_bott-chern_2015}:

\begin{itemize}
	\item \textbf{Real structure:} Consider the following involution on the category $\operatorname{DC}^b_\C$:
	\[
	A=\left(\bigoplus_{p,q\in\Z} A^{p,q},\del_1,\del_2\right)\mapsto \overline{A}:=\left(\bigoplus_{p,q\in\Z}\overline{A^{q,p}},\del_2,\del_1\right),
	\]
	where for some $\C$-vectorspace $V$, the conjugate space $\overline{V}$ is the same space as a set, but with scalar multiplication twisted by complex conjugation, i.e. $\alpha\cdot_{\overline{V}}v:=\overline{\alpha}\cdot V$.\\
	
	Complex conjugation of forms induces a canonical isomorphism $\sigma:\cA_X\cong\overline{\cA}_X$, i.e., $\cA_X$ is a fixed point of this involution. In particular, by Proposition \ref{characterisation of E_r-isomorphisms} $ii)$, for every shape $S$ occuring with multiplicity $m$ in $\cA_X$, its reflection along the diagonal occurs with the same multiplicity.\\
	
	Note that for any double complex $A$ of $\C$-vector spaces, the involution $A\mapsto \overline A$ interchanges $F_1$ and $F_2$. Therefore, if $A$ has a real structure (i.e. is a fixed point of the involution) we will write $F:=F_1$ and $\overline{F}:=F_2$. In such a case, the Frölicher spectral sequence $S_2$ is completely determined by $S_1$. In particular, if one wants to check whether a map between double complexes with real structure is an $E_r$-isomorphism, it suffices to do so for one of the two Frölicher spectral sequences, provided that the map is compatible with the real structures.
	
	\item \textbf{Dimension:} As $X$ is of complex dimension $n$, the complex $\cA_X$ is concentrated in degrees $(p,q)$ with $n\geq p,q\geq0$. In particular, only shapes that lie in that region can have nonzero multiplicity in $\cA_X$.
\end{itemize}
From now on, let us assume that $X$ is compact.
\begin{itemize}
	\item \textbf{Finite dimensional cohomology:} The complex $\cA_X$ itself need not be finite dimensional. However, Dolbeault cohomology (or alternatively Bott-Chern and Aeppli cohomology) can be shown to be finite dimensional by elliptic theory and thus all zigzags have finite multiplicity by Lemma \ref{Dolbeault, Bott Chern and Aeppli via zigzags}. Note that it suffices to know finite dimensionality for Bott-Chern or Aeppli or Dolbeault cohomology, it is automatically implied for the others by the general theory.

	\item \textbf{Duality:} Let $\mathcal{DA}_X$ denote the `dual complex' of $\cA_X$, given by $\mathcal{DA}^{p,q}_X:=(\cA^{n-p,n-q}_X)^\vee:=\Hom_\C(\cA^{n-p,n-q}_X,\C)$ with differentials $\del^\vee,\delbar^\vee$, defined by $(\del^\vee)^{p,q}:=(\varphi\mapsto (-1)^{p+q+1}\varphi\circ\del^{n-p-1,n-q})$ and similarly for $\delbar^\vee$.\\
	
	By construction and as we know that all zigzag shapes have finite multiplicity, for a zigzag shape occuring with a certain multiplicity in $\cA_X$, the shape obtained by reflection at the antidiagonal $p+q=n$ occurs with the same multiplicity in $\mathcal{DA}_X$.\\
	
	As $X$ is a complex manifold, it is automatically oriented. In particular, integration yields a pairing:
	\begin{align*}
	\cA^{p,q}_X\otimes\cA^{n-p,n-q}_X&\longrightarrow \C\\
	(\alpha,\beta)&\longmapsto\int_X\alpha\wedge\beta
	\end{align*}
	
	This induces maps $\Phi^{p,q}:\cA^{p,q}_X\longrightarrow \mathcal{DA}^{p,q}_X$ and the signs are set up so that it yields a morphism of complexes $\Phi:\cA_X\longrightarrow \mathcal{DA}_X$. Serre duality (\cite[thm. 4]{serre_theoreme_1955}) implies that this map is an $E_1$-isomorphism. Thus, by Proposition \ref{characterisation of E_r-isomorphisms}, every zigzag shape occurs with the same multiplicity in $\cA_X$ and $\mathcal{DA}_X$.
	
	\item \textbf{Conectedness:} The shape $\{(0,0)\}$ (and therefore, by duality, also the shape $\{(n,n)\}$) has multiplicity $\#\pi_0(X)$, as functions satisfying $df=0$ are constant on each connected component.
	
	\item \textbf{Only dots and squares in the corners:} For $\omega$ a function or an $(n-1,0)$ form, $\del\delbar\omega=0$ implies $\del\omega=0$ (This follows from the maximum principle for pluriharmonic functions and Stokes' theorem, see e.g. \cite[p. 7f.]{mchugh_bott-chern-eppli_2017}). Combining this with the two dualities, one sees that the implications 
	\[P\in S\text{ and }\mult_S(\cA_X)\neq 0\Longrightarrow S=\{P\}\]
	hold for $P\in\{(0,0),(n,0),(0,n),(n,n)\}$ and zigzag shapes $S$.
	
	\item \textbf{Disjoint unions and direct products:} Given two compact complex manifolds $X,Y$, there is a canonical identification $\cA_{X\sqcup Y}=\cA_X\oplus\cA_Y$ and the natural map $\cA_X\otimes\cA_Y\rightarrow \cA_{X\times Y}$ is an $E_1$-quasi isomorphism. The first statement is clear and the second is the Künneth-formula for Dolbeault-cohomology (see e.g. \cite{griffiths_principles_1978}). In particular, the multiplicities of zigzags in $\cA_{X\times Y}$ can be computed via Proposition \ref{multiplication rules for shapes}. \\
\end{itemize}
\begin{rem}In dimension $2$, more restrictions are known: In fact, for compact complex surfaces, the Frölicher spectral sequence degenerates (see \cite[thm. IV.2.7]{barth_compact_1984}). Thus, for $n=2$,  $\mult_S(\cA_X))=0$ for all even zigzag shapes $S$. Also, as a consequence of \cite[thm. IV.2.6.]{barth_compact_1984}, a compact complex surface $S$ either satisfies the $\del\delbar$-lemma or exactly two zigzags of length $3$ occur in $\cA_S$ (which have shapes $S_1^{0,0}$ and $S^{2,2}_3$) 
\end{rem}

Even though this was little more than a restatement of known results, we can now draw the first (apparently new) consequences:\\

The duality along the antidiagonal implies the following strengthened version of Serre-duality. 
\begin{cor}\label{duality}
Let $H$ be any linear functor $DC_\C^b$ to vector spaces that maps squares to zero. Then the integration pairing induces a canonical isomorphism

\[
H(\cA_X)\cong H(\mathcal{DA}_X)
\]
In particular,
\[
E_r^{p,q}(X)\cong (E_r^{n-p,n-q}(X))^\vee
\]
for all higher pages of the Fr\"olicher spectral sequence.
\end{cor}

\begin{proof}
	The general statement is a direct consequence of Prop. \ref{linear functors and quasi-isomorphisms}. Note that $E_r^{p,q}(\mathcal{DA})\cong E_r^{n-p,n-q}(\cA_X)^\vee$, since taking duals is exact, so it commutes with taking cohomology.
\end{proof}

\begin{rem}
Cor. \ref{duality} also allows to rederive the duality between $H_{BC}(X)$ and $H_{A}(X)$ from Serre duality, without the need for Bott-Chern and Aeppli-Laplacians. If one is only interested in the higher pages of the Fr\"olicher spectral sequence, one can use directly the second part of the argument and classical Serre duality for $E_1$ and does not need to resort to Prop. \ref{linear functors and quasi-isomorphisms}. The corresponding equality of the dimensions one gets for the spaces on the antidiagonals of the $r$-th page, i.e. $\dim E_r^k=\dim E_r^{2n-k}$ for $E_r^k:=\bigoplus_{p+q=k}E_r^{p,q}$, has been shown by Dan Popovici in \cite{popovici_adiabatic_2019-2} using analytic techniques.
\end{rem}

\begin{thm}
	Let $X$ be an $n$-dimensional compact complex manifold. There is a functorial $3$-space decomposition,  orthogonal with respect to the intersection form, 
	\[
	H^n_{dR}(X)=H_{\delbar}^{n,0}(X)\oplus H^{mid}(X)\oplus H_\del^{0,n}(X),
	\]
	where 
	\[
	H^{mid}(X):=\frac{\ker d\cap (\cA_X^{n-1,1}\oplus...\oplus\cA_X^{1,n-1})}{\im d\cap(\cA_X^{n-1,1}\oplus...\oplus\cA_X^{1,n-1})}\subseteq H^n_{dR}(X)
	\]
\end{thm}

\begin{proof}
	First of all, if there is such a decomposition, it is by definition functorial, since the maps from the three spaces on the right to de Rham cohomology are induced by the identity and it is orthogonal for bidegree reasons. In order to see that the maps are injective and their image spans the whole space, one chooses a decomposition and counts zigzags: To the de Rham cohomology all odd-zigzags with most components on the $n$-th antidiagonal contribute. On the other hand, to $H_{\delbar}^{n,0}(X)$, only zigzags with nonzero component in $(n,0)$ and zero component in $(n,1)$ contribute. By the `only dots and squares in the corners' statement, these are only dots. Similarly, the only contribution to $H_{\del}^{0,n}(X)$ comes from dots in degree $(0,n)$. Furthermore, $H^{mid}$, considered as a functor from bounded double complexes to vector spaces which is compatible with direct sums, vanishes on squares and all zigzags except those of shape $S_n^{p,q}$ where either $p+q\leq k$ and $S_n^{p,q}\subseteq \{(n-1,1),...,(1,n-1)\}$ or $p+q\geq k$ and $S_n^{p,q}\cap \{(n-1,1),...,(1,n-1)\}\neq \emptyset$. In particular, the pairwise intersections of the three spaces are $\{0\}$. Finally, again by the dimension restrictions and the `only dots and squares in the corners' statement, if a shape $S_n^{p,q}$ contains $(n,0)$ (resp. $(0,n)$) and an element in $\{(n-1,1),...,(1,n-1)\}$, it has multiplicity zero. So the three spaces also span $H_{dR}^n(X)$.
\end{proof}

We now calculate $\cA_X$ up to $E_1$-isomorphism for several examples of compact complex manifolds:\\

\textbf{$\del\delbar$-manifolds:}\\
Assume $\cA_X$ satisfies the $\del\delbar$-lemma (e.g., $X$ Kähler or more generally in class $\mathcal{C}$, i.e., bimeromorphic to a Kähler manifold). By Proposition \ref{generalised DGMS}, $\cA_X$ is then a direct sum of squares and dots and for any $k$ there is the Hodge decomposition $H^k_{dR}(X)=F^p\cap\overline{F}^q$, so we consider $H_{dR}(X):=\bigoplus_{k\in\Z}H^k_{dR}(X)$ as a bigraded vector space and we make it into a double complex by declaring the differentials to be $0$. The multiplicities of the dots coincide with the Hodge numbers, which also coincide with the dimensions of the $F^p\cap \overline{F}^q$. Therefore
\[
(\cA_X,\del_1,\del_2)\simeq_1 (H_{dR}(X),0,0).
\]
E.g., if $S_g$ denotes a Riemann surface of genus $g\geq 0$:
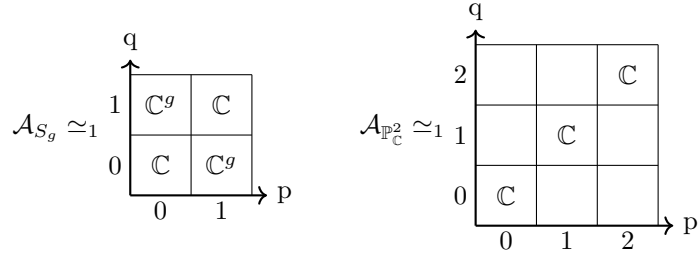
\begin{center}
\begin{tabular}{cc}
	$\cA_{S_g}\simeq_1$\parbox{3cm}{\begin{tikzpicture}[scale=0.4]
	
	\draw[->, thick] (0,0) -- (4.5,0) node[anchor=west] {p};
	\draw (0,2) -- (4,2);
	\draw (0,4) -- (4,4);
	
	\draw[->, thick] (0,0) -- (0,4.5) node[anchor=south]{q};
	\draw (2,0) -- (2,4);
	\draw (4,0) -- (4,4);
	
	\draw (1,-.5) node {0};
	\draw (3,-.5) node {1};
	
	\draw (-.5,1) node {0};
	\draw (-.5,3) node {1};
	
	\node (00) at (1,1) {$\C$};
	\node (11) at (3,3) {$\C$};
	\node (12) at (1,3) {$\C^g$};
	\node (21) at (3,1) {$\C^g$};
	
	\end{tikzpicture}}&	
	
	$\cA_{\Pro^2_\C}\simeq_1$\parbox{3.5cm}{\begin{tikzpicture}[scale=0.4]
		
		\draw[->, thick] (0,0) -- (6.5,0) node[anchor=west] {p};
		\draw (0,2) -- (6,2);
		\draw (0,4) -- (6,4);
		\draw (0,6) -- (6,6);
		
		\draw[->, thick] (0,0) -- (0,6.5) node[anchor=south]{q};
		\draw (2,0) -- (2,6);
		\draw (4,0) -- (4,6);
		\draw (6,0) -- (6,6);

		\draw (1,-.5) node {0};
		\draw (3,-.5) node {1};
		\draw (5,-.5) node {2};
		
		\draw (-.5,1) node {0};
		\draw (-.5,3) node {1};
		\draw (-.5,5) node {2};
		
		\node (00) at (1,1) {$\C$};
		\node (00) at (3,3) {$\C$};
		\node (00) at (5,5) {$\C$};
		\end{tikzpicture}}
\end{tabular}
\end{center}

\textbf{Calabi-Eckmann manifolds:}\\
In \cite{calabi_class_1953}, for any $\alpha\in\C\backslash\R$ and $u,v\in\Z_{\geq 0}$ a manifold $M_{u,v}^\alpha$ was defined by putting a complex structure on the product $S^{2u+1}\times S^{2v+1}$ such that the projection
\[
S^{2u+1}\times S^{2v+1}\longrightarrow \Pro^{u}_\C\times \Pro^v_\C
\]
is a holomorphic fibre bundle with fibre $T=\C/(\Z+\alpha\Z)$. Explicitly, they can be realised as the quotient
\[
M_{u,v}^{\alpha}=((\C^{u+1}\backslash\{0\})\times(\C^{v+1}\backslash\{0\}))/\sim
\]
where
\[
(x,y)\sim (e^{t}x,e^{\alpha t}y)\qquad \text{for any }t\in\C.
\]

Since the following discussion does not depend on the choice of $\alpha$, we write $M_{u,v}$ for the product $S^{2u+1}\times S^{2v+1}$ equipped with any of these complex structures.

\begin{ex}
	$M_{0,0}$ is a complex torus and $M_{0,v}$ or $M_{u,0}$ are Hopf manifolds.	
\end{ex}
In the following, we assume $u<v$ for simplicity.\\

In \cite{hirzebruch_topological_1978}, Borel computed the first page of the Frölicher spectral sequence for $M_{u,v}$. Numerically, the result reads:
\[
h_{M_{u,v}}^{p,q}=\begin{cases}
1&\text{if } p\leq u\text{ and } q=p,p+1\\
1&\text{if } p> v\text{ and } q=p,p-1\\
0&\text{else.}
\end{cases}
\]

The following is a picture of the $E_1$-page (without differentials) for $u=1,v=2$:

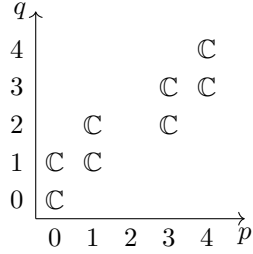
\begin{center}
	\begin{tikzpicture}
	\draw[->] (-0.25,-0.25) -- (2.5,-0.25) node[anchor=north] {$p$};
	\draw[->] (-0.25,-0.25) -- (-0.25,2.5) node[anchor=east] {$q$};
	
	\draw (0,-0.5) node {$0$};
	\draw (0.5,-0.5) node {$1$};
	\draw (1,-0.5) node {$2$};
	\draw (1.5, -0.5) node {$3$};
	\draw (2, -0.5) node {$4$};
	
	\draw (-0.5,0) node {$0$};
	\draw (-0.5,0.5) node {$1$};
	\draw (-0.5,1) node {$2$};
	\draw (-0.5,1.5) node {$3$};
	\draw (-0.5,2) node {$4$};
	
	\draw (0,0) node {$\C$};
	\draw (0.5,0) node {};
	\draw (1,0) node {};
	\draw (1.5, 0) node {};
	\draw (2, 0) node {};
	
	\draw (0,0.5) node {$\C$};
	\draw (0.5,0.5) node {$\C$};
	\draw (1,0.5) node {};
	\draw (1.5, 0.5) node {};
	\draw (2, 0.5) node {};
	
	\draw (0,1) node {};
	\draw (0.5,1) node {$\C$};
	\draw (1,1) node {};
	\draw (1.5, 1) node {$\C$};
	\draw (2, 1) node {};
	
	\draw (0,1.5) node  {};
	\draw (0.5,1.5) node {};
	\draw (1,1.5) node {};
	\draw (1.5, 1.5) node {$\C$};
	\draw (2, 1.5) node {$\C$};
	
	\draw (0,2) node {};
	\draw (0.5,2) node {};
	\draw (1,2) node  {};
	\draw (1.5, 2) node {};
	\draw (2, 2) node {$\C$};
	
	\end{tikzpicture}
\end{center}

Since the underlying topological space of $M_{u,v}$ is a product of spheres, the Betti-numbers are given by
\[
b^i_{M_{u,v}}=\begin{cases}
1&\text{if }i=0,2u+1,2v+1,2(u+v)+2\\
0&\text{else.}
\end{cases}
\]

In particular, the Frölicher spectral sequence degenerates at the second stage and the multiplicities of all zigzags can be determined combinatorially. (For the odd-length zigzags, use that the de Rham cohomology is one-dimensional, so knowing the breakpoints of each filtration individually determines the nonzero bidegree of the associated bigraded.) This results in:
\[
\mult_S(\cA_{M_{u,v}})=\begin{cases}
1&\text{if }S=S_{0}^{0,0},S_{2u+1}^{u,u}\\
1&\text{if }S=S_{1,1}^{p,p+1}\text{ for }0\leq p <u\\
0&\text{in all other cases not determined by duality and real structure.}
\end{cases}
\]
The result on Bott-Chern and Aeppli-cohomology stated in the introduction can be read off from this using Lemma \ref{Dolbeault, Bott Chern and Aeppli via zigzags}.\\

Continuing with the case $u=1$, $v=2$, a picture of all the zigzags in $\cA_{M_{1,2}}$ looks as follows:
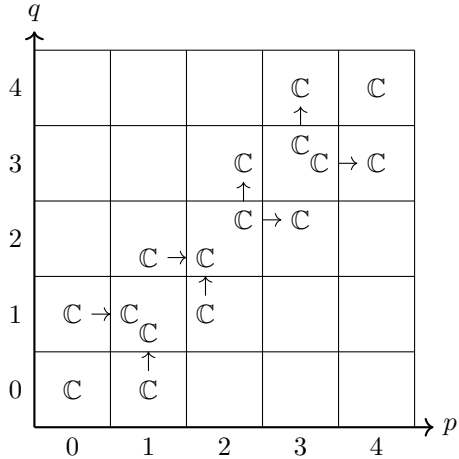
\[
\begin{tikzpicture}[scale=0.5]

\draw[->, thick] (0,0) -- (10.5,0) node[anchor=west] {$p$};
\draw (0,2) -- (10,2);
\draw (0,4) -- (10,4);
\draw (0,6) -- (10,6);
\draw (0,8) -- (10,8);
\draw (0,10) -- (10,10);

\draw[->, thick] (0,0) -- (0,10.5) node[anchor=south]{$q$};
\draw (2,0) -- (2,10);
\draw (4,0) -- (4,10);
\draw (6,0) -- (6,10);
\draw (8,0) -- (8,10);
\draw (10,0) -- (10,10);

\draw (1,-.5) node {0};
\draw (3,-.5) node {1};
\draw (5,-.5) node {2};
\draw (7,-.5) node {3};
\draw (9,-.5) node {4};

\draw (-.5,1) node {0};
\draw (-.5,3) node {1};
\draw (-.5,5) node {2};
\draw (-.5,7) node {3};
\draw (-.5,9) node {4};

\node (00) at (1,1) {$\C$};
\node (01) at (1,3) {$\C$};
\node (10) at (3,1) {$\C$};
\node (11l) at (2.5,3) {$\C$};
\node (11r) at (3,2.5) {$\C$};
\node (12) at (3,4.5) {$\C$};
\node (21) at (4.5,3) {$\C$};
\node (22l) at (4.5,4.5) {$\C$};
\node (22r) at (5.5,5.5) {$\C$};
\node (23) at (5.5,7) {$\C$};
\node (32) at (7,5.5) {$\C$};
\node (34) at (7,9) {$\C$};
\node (43) at (9,7) {$\C$};
\node (33l) at (7,7.5) {$\C$};
\node (33r) at (7.5,7) {$\C$};
\node (44) at (9,9) {$\C$};

\draw[->] (10) to (11r);
\draw[->] (01) to (11l);
\draw[->] (12) to (22l);
\draw[->] (21) to (22l);
\draw[->] (22r) to (23);
\draw[->] (22r) to (32);
\draw[->] (33l) to (34);
\draw[->] (33r) to (43);
\end{tikzpicture}
\]
\begin{rem}
	For the cases $u$ or $v$ equal to $0$, this coincides with a description of the zigzags in the Dolbeault double complex obtained (under a conjecture) by G. Kuperberg on MO.\footnote{\url{https://mathoverflow.net/questions/25723/}}
\end{rem}
\begin{rem}
	For $u=v>0$, it seems one cannot just argue numerically since in this case $b^{2u+1}_{M_{u,u}}=2$ and so it does not suffice to know in which degree the Hodge filtration and its conjugate jump individually. More precisely, the methods used here only allow to say that there are either two zigzags with shapes $S_{2u+1}^{u,u+1},S_{2u+1}^{u+1,u}$ or two zigzags with shapes $S_{2u+1}^{u,u},S_{2u+1}^{u+1,u+1}$.
\end{rem}

\textbf{Nilmanifolds}\\
A rich and well-studied class of examples of compact complex manifolds which are very accesible from the computational point of view are complex nilmanifolds:\footnote{In addition to the references we give in the text, we refer to \cite[ch. 1.7.2 and ch. 3]{angella_cohomological_2013} for a survey of the literature} They are quotients $X=G/\Gamma$ of a nilpotent (real) Lie groups with left invariant complex structure by a lattice. They have the distinguished finite dimensional subcomplex $\cA_X^{inv}\subseteq\cA_X$ of invariant differential forms. In many (and conjecturally all) cases (\cite{console_dolbeault_2001}) the inclusion $\cA_X^{inv}\subseteq\cA_X$ is an $E_1$-isomorphism. In particular, this holds in complex dimensions up to $3$ (\cite{fino_dolbeault_2019}).\\

One can construct an $n$-dimensional nilmanifold starting from the a tuple 
\[
(\mathfrak{g}, J, L),
\] where 
\begin{itemize}
	\item $\mathfrak{g}$ is a nilpotent real Lie-algebra of real dimension $2n$ which has a $\Q$-structure, i.e. there is a Lie-algebra $\mathfrak{g}_\Q$ over the rationals s.t. $\mathfrak{g}=\mathfrak{g}_\Q\otimes\R$. Concretely, this means there exists a basis $(e_1,...e_{2n})$ with rational structure constants $c_{i,j}^k\in\Q$ (defined by $[e_i,e_j]=\sum_{k=1}^{2n}c_{i,j}^ke_k))$,
	\item $J:\mathfrak{g}\longrightarrow\mathfrak{g}$ is an linear map which is an almost complex structure (i.e. $J^{-1}=-1$) and integrable ($[Jx,Jy]=J[Jx,y]+J[x,Jy]+[x,y]~\forall x,y\in\mathfrak{g}$),
	\item $L$ is a lattice in $\mathfrak{g}_\Q$.
\end{itemize}
Via the exponential map, the pair $(\mathfrak{g},L)$ is turned into a simply connected nilpotent Lie group $G$ with lattice $\Gamma$  (see e.g. \cite[Thm. 2.12.]{raghunathan_discrete_1972}) and $J$ induces a left invariant complex structure on $G$. The space $G/\Gamma$ is then a nilmanifold. $J$ induces a decomposition $\mathfrak{g}_\C^\vee=\mathfrak{g}^{1,0}\oplus\mathfrak{g}^{0,1}$ of $\mathfrak{g}^\vee_\C=\Hom(\mathfrak{g},\C)$ into $i$ and $-i$ eigenspaces. The left invariant $(p,q)$ forms are identified with the spaces $\Lambda^p\mathfrak{g}^{p,0}\otimes\Lambda^q\mathfrak{g}^{0,q}$ and the exterior differential with 
the one defined in degree one by $d\omega(X,Y):=-\omega([X,Y])$.\\

For a concrete example, let us consider the Lie-algebra $\mathfrak{h}_9$, i.e., the 6-dimensional real vector space with basis $e_1,...,e_6$ and Lie bracket given by:\footnote{This Lie algebra and background information can be found e.g. in \cite{andrada_classification_2011}.}
\[
[e_1,e_2]=e_4,\qquad [e_1,e_3]=[e_2,e_4]=-e_6
\]
where the nonmentioned brackets are defined by antisymmetry or $0$. We endow this with an almost complex structure  defined by:
\[
Je_1=e_2,\qquad Je_3=e_4,\qquad Je_5=e_6
\]
A $\Q$-structure is given by the $\Q$-span of the $e_i$ and a lattice $L$ by their $\Z$-span.\\

Let us denote the dual basis vectors to the $e_i$ by $e^i$. A basis of $\mathfrak{g}^{1,0}$, the $i$ eigenspace of $J$ in the complexified Lie algebra $\mathfrak{g}_\C$, is obtained by setting:
\[
\omega^1:= e^1-ie^2,\qquad \omega^2:= e^3-ie^4,\qquad \omega^3:= e^5-ie^6
\]
Plugging in the definitions, one checks that the differential is given on $\mathfrak{g}^{1,0}$ by the following rules:
\[
d\omega^1=0,\qquad d\omega^2=\frac{1}{2}\bar{\omega}^1\wedge\omega^1,\qquad d\omega^3=\frac{i}{2}(\omega^1\wedge\bar{\omega}^2+\bar{\omega}^1\wedge\omega^2)
\]
This and the proof of Theorem \ref{decomposition of double complexes}, or Prop. \ref{multiplicities and cohomology}, allow the computation of the $E_1$-isomorphism type of the Lie-algebra double complex (and hence of $\cA_X$ for the corresponding nilmanifold $X$). We spare the reader the calculation and just give the result. All zigzag shapes whose multiplicity is not determined via duality and real structure from the ones listed here have multiplicity zero.
\begin{center}
	\begin{tabular}{l|r}
		Shape of zigzag & Generators\\
		\hline
		&\\
		$S_0^{0,0}$& $1$\\
		$S_1^{1,0}$&$\omega^1$\\
		$S_1^{0,0}$& $\omega^2,\omega^3, \overline{\omega}^2, \overline{\omega}^3$\\
		$S_2^{2,0}$&$\omega^1\wedge\omega^2$\\
		$S_2^{1,0}$&$\omega^1\wedge\omega^3, \omega^1\wedge\overline{\omega}^3$\\
		$S_2^{1,1}$&$\omega^2\wedge\overline{\omega}^1-\omega^1\wedge\overline{\omega}^2,i\omega^2\wedge\overline{\omega}^2+\omega^1\wedge\overline{\omega}^3-\omega^3\wedge\overline{\omega}^1$\\
		$S_2^{0,0}$&$\omega^2\wedge\omega^3,\omega^3\wedge\overline{\omega}^2+\omega^2\wedge\overline{\omega}^3,\overline{\omega}^2\wedge\overline{\omega}^3$\\
		$S_3^{3,0}$&$\omega^1\wedge\omega^2\wedge\omega^3$\\
		$S_3^{2,1}$&$\omega^1\wedge\omega^2\wedge\overline{\omega}^3$\\
		$S_3^{2,2}$&$\omega^3\wedge\overline{\omega}^3,\omega^2\wedge\overline{\omega}^3-\omega^3\wedge\overline{\omega}^2$
	\end{tabular}
\end{center}

The following encodes this information (without explicit generators) in a diagram. Again, zigzags determined by duality and real structure are not drawn.
\begin{center}
	\begin{tikzpicture}[scale=0.65]
	\draw[->, thick] (0,0) -- (10.5,0) node[anchor=west]{$p$};
	\draw (0,2) -- (10,2);
	\draw (0,5) -- (10,5);
	\draw (0,8) -- (10,8);
	\draw (0,10) -- (10,10);
	\draw[->, thick] (0,0) -- (0,10.5) node[anchor=south]{$q$};
	\draw (2,0) -- (2,10);
	\draw (5,0) -- (5,10);
	\draw (8,0) -- (8,10);
	\draw (10,0) -- (10,10);
	
	\draw (1,-.5) node {$0$};
	\draw (3.5,-.5) node {$1$};
	\draw (6.5,-.5) node {$2$};
	\draw (9,-.5) node {$3$};
	
	\draw (-.5,1) node {$0$};
	\draw (-.5,3.5) node {$1$};
	\draw (-.5,6.5) node {$2$};
	\draw (-.5,9) node {$3$};

	\node (00) at (1,1) {$\C$};
	\node (10l) at (2.5,1) {$\C^2$};
	\node (10r) at (4,1) {$\C$};
	\node (20l) at (5.5,1) {$\C$};
	\node (20m) at (6.5,1) {$\C$};
	\node (20r) at (7.5,1) {$\C$};
	\node (30) at (9,1) {$\C$};
	
	\node (01) at (1,2.5) {$\C^2$};
	\node (11bl) at (2.5,2.5) {$\C^2$};
	\node (11) at (3,3) {$\C^2$};
	\node (11m) at (3.5,3.5) {$\C$};
	\node (11br) at (4,2.5) {$\C$};
	\node (11tr) at (4.5,4.5) {$\C^2$};
	\node (21bl) at (5.5,2.5) {$\C$};
	\node (21m) at (6.5,3.5) {$\C$};
	\node (21r) at (7.5,3) {$\C$};
	\node (21tl) at (6,4.5) {$\C^2$};
	
	\node (02) at (1,6.5) {$\C$};
	\node (12m) at (3.5,6.5) {$\C$};
	\node (12r) at (4.5,6) {$\C^2$};
	
	\draw[->] (10l) to (11bl);
	\draw[->] (01) to (11bl);
	
	\draw[->] (20l) to (21bl);
	\draw[->] (11br) to (21bl);
	
	\draw[->] (20m) to (21m);
	\draw[->] (11m) to (21m);
	\draw[->] (11m) to (12m);
	\draw[->] (02) to (12m);
	
	\draw[->] (11tr) to (21tl);
	\draw[->] (11tr) to (12r);
	\end{tikzpicture}
\end{center}
We conclude this example by remarking that this particular nilmanifold is interesting because it admits an endomorphism which is not strictly compatible with the Hodge filtration on the de Rham cohomology.\footnote{Recall that a linear map of filtered vector spaces $\varphi:(V,F)\longrightarrow (\widetilde{V},\widetilde{F})$ is called strict if $\varphi F^p=\widetilde{F}^p\cap\im\varphi$ for all $p$.} In fact, if $\varphi:\mathfrak{g}\longrightarrow\mathfrak{g}$ is defined by
\[
e_i\longmapsto\begin{cases}	e_3-e_6& i=1\\
e_4+e_5& i=2\\
0&\text{else.}
\end{cases}
\]
This is a map of Lie-algebras compatible with the complex structure and maps $L$ to $L$ and therefore induces a holomorphic map $\tilde\varphi:G/\Gamma\rightarrow G/\Gamma$. The cohomology can be computed by left invariant forms and the induced morphism $\varphi^*:H^1_{dR}(G/\Gamma,\C)\rightarrow H^1_{dR}(G/\Gamma,\C)$ is determined by the
dual map $\varphi^\vee:\mathfrak{g}_\C^\vee\rightarrow \mathfrak{g}_\C^\vee$, which in turn is determined by its values on $\mathfrak{g}^{1,0}$, namely:

\begin{eqnarray*}
	\omega^1&\longmapsto&0\\
	\omega^2&\longmapsto&\omega^1\\
	\omega^3&\longmapsto&i\omega^1
\end{eqnarray*}

Using this, one checks
\[
\varphi^*F^1H^1_{dR}(G/\Gamma,\C)=\{0\}\neq\langle[\omega^1]\rangle=F^1H^1_{dR}(G/\Gamma,\C)\cap\varphi^*H^1_{dR}(G/\Gamma,\C),
\]
i.e., $\varphi^*$ is not strict. This phenomenon can be seen as an incarnation of the fact that the decomposition of double complexes into indecomposables is not functorial (although certainly this theory is not needed to give this example). While it is certainly the expected behaviour for a general morphism between general complex manifolds (in sharp contrast to $\del\delbar$-manifolds, where morphisms are automatically strict), it seems like no example of a non-strict morphism between \textit{compact} manifolds has appeared in the literature before. \\

\textbf{Blowups:}\\
For a double complex $A$ and an integer $i\in\Z$, denote by $A[i]$ the shifted double complex, which has underlying graded $A[i]^{p,q}=A^{p-i,q-i}$. In \cite{stelzig_double_2019} the following results were shown (with explicit maps):
\begin{thm}\label{projective bundles, modifications, blowups}Let $X$ be a compact, connected manifold of dimension $n$.
	\begin{enumerate}
		\item \textbf{Projective bundles:} Let $\V\longrightarrow X$ be a vector bundle of rank $m+1$ and $\Pro(\V)\longrightarrow X$ the associated projective bundle. Then
		\[
		\cA_{\Pro(\V)}\simeq_1\cA_{\Pro^m_\C}\otimes\cA_X\simeq_1 \bigoplus_{i=0}^{m}\cA_X[i].
		\]
		\item \textbf{Modifications:} For a surjective holomorphic map $f:Y\longrightarrow X$ with $Y$ compact connected of dimension $n$, one has
		\[
		\cA_Y\simeq_1\cA_X\oplus \cA_Y/f^*\cA_X.
		\]
		
		\item \textbf{Blowups:} For a submanifold $Z\subseteq X$ of codimension $r\geq 2$, let $\widetilde{X}$ be the blowup of $X$ in $Z$. Then
		\[
		\cA_{\widetilde{X}}\simeq_1 \cA_X\oplus\bigoplus_{i=1}^{r-1}\cA_Z[i].
		\]
	\end{enumerate}
\end{thm}
Pictorially, these results can be understood as follows, where in the first case, the small square represents $\cA_X$ and there are $m$ copies of it and in the second case, the big square represents $\cA_X$ and the smaller ones $\cA_Z$ and there are $r-1$ copies.
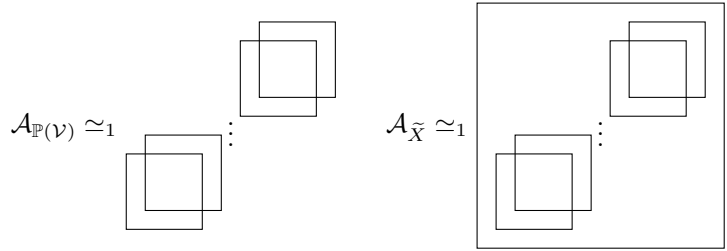
\begin{center}
\begin{tabular}{cc}
	$\cA_{\Pro(\V)}\simeq_1$	\parbox{3cm}{\begin{tikzpicture}[scale=0.5]
	\draw (0,0) rectangle ++(2,2);
	\draw (0.5,0.5) rectangle ++(2,2);
	\node at (2.75,2.75){$\vdots$};
	\draw (3,3)rectangle ++(2,2);
	\draw (3.5,3.5) rectangle ++(2,2);
	\end{tikzpicture}}&

$\cA_{\widetilde{X}}\simeq_1$	\parbox{3cm}{\begin{tikzpicture}[scale=0.5]
	\draw (-0.5,-0.5) rectangle +(6.5,6.5);
	\draw (0,0) rectangle ++(2,2);
	\draw (0.5,0.5) rectangle ++(2,2);
	\node at (2.75,2.75){$\vdots$};
	\draw (3,3)rectangle ++(2,2);
	\draw (3.5,3.5) rectangle ++(2,2);
	\end{tikzpicture}}
\end{tabular}
\end{center}

One can use the theory developed so far to refine several corollaries in \cite{rao_dolbeault_2019}, \cite{yang_bottchern_2020}, \cite{angella_note_2020}. A key point in the proofs is (as in the cited works) the fact that every bimeromorphic map can be factored as a sequence of blowups and blowdowns (\cite{abramovich_torification_2002}, \cite{wlodarczyk_toroidal_2003}) and can be deduced easily from the above pictorial description of the double complex of a blowup.

\begin{cor}\label{bimeromorphically invariant multiplicities}
	The numbers $\mult_{S}(X)$ are bimeromorphic invariants whenever $S$ is a zigzag shape and there is a point $(a,b)\in S$ s.t. 
	\begin{enumerate}
		\item $a\in\{0,n\}$ or $b\in\{0,n\}$ or
		\item $(a,b)\in\{(1,1),(1,n-1),(n-1,1),(n-1,n-1)\}$ and $S$ is not a dot.	
	\end{enumerate}
\end{cor}
\begin{proof}
	By the weak factorisation theorem, it is enough to show that these numbers do not change when passing from $\cA_X$ to the double complex of a blowup $\cA_{\widetilde{X}}$. The statement is then a direct consequence of the pictorial description and the `only squares and dots in the corners' for $Z$ (the center of the blowup).
\end{proof}

In particular, this corollary gives a unified proof of the known fact that $H_{\delbar}^{p,0}(X)$, $H_{\delbar}^{0,p}$ and $H_{BC}^{p,0}$ are bimeromorphic invariants, but also shows that the same holds for $E_r^{p,0}(X)$ and $E_r^{0,p}$ for any $r\geq 1$. Maybe somewhat surprisingly, it shows that bimeromorphic invariance also holds for certain numbers away from the boundary region, such as the refined Betti-numbers $b_3^{2,2}$ or $b_n^{n-1,p}$ for $p>1$ or the the dimension of the image of differentials in the Frölicher spectral sequence ending in degree $(n-1,1)$\\


Let us say a property $(P)$ of compact complex manifolds is stable under restriction if any complex submanifold of a compact complex manifold satisfying $(P)$ also satisfies $(P)$. 

\begin{cor}\label{bimeromorphical invariance and submanifold inheritance}
The following properties are bimeromorphic invariants of compact complex manifolds if and only if they are stable under restriction.
\begin{itemize}
		\item The Frölicher spectral sequence degenerates at stage $\leq r$.
		\item The $k$-th de Rham cohomology groups satisfy a Hodge decomposition \[H_{dR}^k=\bigoplus_{p+q=k} F^p\cap \overline{F}^q\] for all $k$.
		\item The $\del\delbar$-lemma holds.
	\end{itemize}
\end{cor}
In particular, the degeneracy of the Frölicher spectral sequence on the first page is a bimeromorphic invariant in dimension $\leq 4$, because it always holds for surfaces. In general, this corollary only implies that in dimension $n$, degeneracy at page $n-2$ is bimeromorphically invariant. The second and third properties are implied to be bimeromorphic invariants in dimensions $\leq 3$.\\

To my knowledge, it is in general an open problem whether or not all three properties hold for submanifolds.
\begin{proof}
Recall that by Corollary \ref{generalised DGMS}, the first statement is equivalent to only odd zigzags and even zigzags up to a certain length having nonzero multiplicity, the second to `the only odd zigzags are dots' and the third to the first and second at the same time. By the weak factorisation theorem, it suffices to consider the situation of a blow up $\widetilde{X}$ of $X$ along a submanifold $Z\subseteq X$ of codimension $\geq 2$. Looking at the pictorial description given above, it is clear that any of these three statements holds for $\cA_{\widetilde{X}}$ if and only if it holds for $\cA_Z$ and $\cA_X$. Since one can always increase the codimension of $Z$ without changing any of the desired properties of the ambient manifold by considering $Z\times\{\operatorname{pt}\}\subseteq X\times Y$ for appropriate $Y$ (e.g. satisfying the $\del\delbar$-Lemma), this implies the statement.\footnote{In the preprint version of the article, the proof was left to the reader. Recently, a more detailed version of the proof for the $\del\delbar$-case, using Corollary \ref{non-Kaehlerness-degrees of blowups} below, has been spelled out in the expository note \cite{meng_heredity_2019}.}
\end{proof}

In \cite{yang_bottchern_2020}, it was shown that the non-Kählerness degrees $\Delta^k(X):=\Delta^k(\cA_X)$ are bimeromorphic invariants in dimension $\leq 3$ and it was asked wheter the same holds in dimensions $n\geq 4$. We may answer this question negatively as follows:\\

Take a Hopf surface $H$ and consider the blowup $H'$ of $H\times\{\operatorname{pt}\}$ in $H\times \Pro^2$. It is instructive to compute the double complexes of all spaces involved in this example up to $E_1$-isomorphism, using the pictorial interpretation of Theorem \ref{projective bundles, modifications, blowups} and the explicit description of the $E_1$-isomorphism type of $\cA_H$, so let us do this:\\

Taking the product with $\Pro^2$ yields a direct sum of three (shifted) copies:
\[
\cA_{H\times\Pro^2}\simeq_1 \cA_{H}\oplus\cA_{H}[1]\oplus \cA_H[2]
\]
Taking the blowup along $H\times\{\operatorname{pt}\}$ doubles the middle copy:
\[
\cA_{H'}\simeq_1 \cA_{H}\oplus\cA_{H}[1]^{\oplus 2}\oplus \cA_H[2]
\]

The double complex of $H$ has been computed as
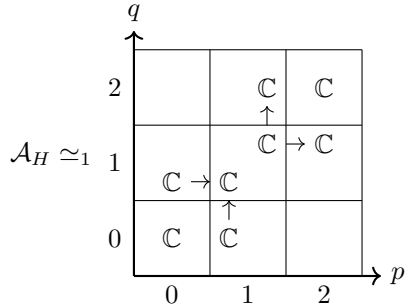
\[
\cA_H\simeq_1\parbox{3cm}{\begin{tikzpicture}[scale=0.5]
	
	\draw[->, thick] (0,0) -- (6.5,0) node[anchor=west] {$p$};
	\draw (0,2) -- (6,2);
	\draw (0,4) -- (6,4);
	\draw (0,6) -- (6,6);
	
	\draw[->, thick] (0,0) -- (0,6.5) node[anchor=south]{$q$};
	\draw (2,0) -- (2,6);
	\draw (4,0) -- (4,6);
	\draw (6,0) -- (6,6);
	
	\draw (1,-.5) node {0};
	\draw (3,-.5) node {1};
	\draw (5,-.5) node {2};
	
	\draw (-.5,1) node {0};
	\draw (-.5,3) node {1};
	\draw (-.5,5) node {2};
	
	\node (00) at (1,1) {$\C$};
	\node (01) at (1,2.5) {$\C$};
	\node (10) at (2.5,1) {$\C$};
	\node (11l) at (2.5,2.5) {$\C$};
	\node (11r) at (3.5,3.5) {$\C$};
	\node (12) at (3.5,5) {$\C$};
	\node (21) at (5,3.5) {$\C$};
	\node (22) at (5,5) {$\C$};

	\draw[->] (10) to (11l);
	\draw[->] (01) to (11l);
	\draw[->] (11r) to (12);
	\draw[->] (11r) to (21);
	\end{tikzpicture}}
\]
In particular, one may now count zigzags and see that 
\begin{align*}\Delta^4(H\times \Pro^2)=\Delta^4(H)+\Delta^2(H)+\Delta^0(H)=\Delta^2(H)&=2\\
&<4=2\Delta^2(H)=\Delta^4(H').\end{align*}

In fact, one has the following more general result, proven by similar means as a consequence of Theorem \ref{projective bundles, modifications, blowups} and the calculation of the non-Kählerness degrees in section \ref{sec: Cohomologies and Multiplicities}:

\begin{cor}\label{non-Kaehlerness-degrees of blowups}
	Given a blowup $\widetilde{X}$ of a compact complex manifold $X$ along a submanifold $Z$, the non-Kählerness degrees satisfy 
\[
\Delta^k(\widetilde{X})\geq\Delta^k(X)
\]
and equality holds for $k=0,1,2,2n-2,2n-1,2n$ (and $k=3$ if $n=3$). Equality holds for all $k$ if and only if $Z$ satisfies the $\del\delbar$-lemma. In particular, they are generally not bimeromorphic invariants in dimensions $n\geq 4$.
\end{cor}

\textbf{Acknowledgments:} This article was written at LMU München. It is an extended and streamlined version of a part of my PhD-thesis \cite{stelzig_double_2018} written at WWU Münster within the SFB 878, and supervised by Christopher Deninger and Jörg Schürmann. I gratefully acknowledge this support. In addition, I would like to thank Mikhail Khovanov and You Qi for sending me \cite{khovanov_spectral_2007}, \cite{qi_quiver_2010}, \cite{khovanov_faithful_2020} and for encouraging conversations, Daniele Angella and the University of Florence and Dieter Kotschick and the LMU for a chance to present an early version of some of the results, their hospitality and several interesting remarks and conversations. For such remarks and conversations I am also grateful to Joana Cirici, Lutz Hille, Matthias Kemper and Robin Sroka.

\bibliographystyle{acm}

\begin{thebibliography}{10}

\bibitem{abramovich_torification_2002}
{\sc Abramovich, D., Karu, K., Matsuki, K., and W{\l}odarczyk, J.}
\newblock {Torification and Factorization of Birational Maps}.
\newblock {\em Journal of the American Mathematical Society 15}, 3 (2002),
  531--572.

\bibitem{andrada_classification_2011}
{\sc Andrada, A., Barberis, M.~L., and Dotti, I.}
\newblock {Classification of Abelian Complex Structures on 6-Dimensional
  {{Lie}} Algebras}.
\newblock {\em Journal of the London Mathematical Society 83}, 1 (Feb. 2011),
  232--255.

\bibitem{angella_cohomological_2013}
{\sc Angella, D.}
\newblock {\em {Cohomological {{Aspects}} in {{Complex Non}}-{{K{\"a}hler
  Geometry}}}}.
\newblock No.~2095 in {Lecture {{Notes}} in {{Mathematics}}}. {Springer}, Nov.
  2013.

\bibitem{angella_cohomologies_2013-1}
{\sc Angella, D.}
\newblock {The Cohomologies of the {{Iwasawa}} Manifold and of Its Small
  Deformations}.
\newblock {\em Journal of Geometric Analysis 23}, 3 (July 2013), 1355--1378.

\bibitem{angella_bott-chern_2015}
{\sc Angella, D.}
\newblock {On the {Bott-Chern} and {Aeppli} cohomology}.
\newblock {\em Bielefeld Geometry \& Topology Days\/} (July 2015).

\bibitem{angella_hodge_2018}
{\sc Angella, D.}
\newblock {Hodge numbers of a hypothetical complex structure on $S^6$}.
\newblock {\em Differential Geometry and its Applications 57\/} (Apr. 2018),
  105--120.

\bibitem{angella_bott-chern_2017}
{\sc Angella, D., and Kasuya, H.}
\newblock {Bott-{{Chern}} Cohomology of Solvmanifolds}.
\newblock {\em Annals of Global Analysis and Geometry 52}, 4 (Dec. 2017),
  363--411.

\bibitem{angella_note_2020}
{\sc Angella, D., Suwa, T., Tardini, N., and Tomassini, A.}
\newblock {Note on {{Dolbeault}} Cohomology and {{Hodge}} Structures up to
  Bimeromorphisms}.
\newblock {\em Complex Manifolds 7}, 1 (Sept. 2020), 194--214.

\bibitem{angella_$partialoverlinepartial$_2013}
{\sc Angella, D., and Tomassini, A.}
\newblock {On the $\partial\overline{\partial}$ -Lemma and Bott-Chern
  cohomology}.
\newblock {\em Inventiones mathematicae 192}, 1 (Apr. 2013), 71--81.

\bibitem{auslander_representation_1974}
{\sc Auslander, M.}
\newblock {Representation {{Theory}} of {{Artin Algebras II}}}.
\newblock {\em Communications in Algebra 1}, 4 (Jan. 1974), 269--310.

\bibitem{azumaya_corrections_1950}
{\sc Azumaya, G.}
\newblock {Corrections and Supplementaries to My Paper Concerning
  {{Krull}}-{{Remak}}-{{Schmidt}}'s Theorem}.
\newblock {\em Nagoya Mathematical Journal 1\/} (1950), 117--124.

\bibitem{barth_compact_1984}
{\sc Barth, W., Peters, C., and Ven, A.}
\newblock {\em {Compact Complex Surfaces}}.
\newblock {Springer}, 1984.

\bibitem{calabi_class_1953}
{\sc Calabi, E., and Eckmann, B.}
\newblock {A {{Class}} of {{Compact}}, {{Complex Manifolds Which}} Are Not
  {{Algebraic}}}.
\newblock {\em Annals of Mathematics 58}, 3 (1953), 494--500.

\bibitem{carlsson_parametrized_2019}
{\sc Carlsson, G., de~Silva, V., Kali\v{s}nik, S., and Morozov, D.}
\newblock {Parametrized Homology via Zigzag Persistence}.
\newblock {\em Algebraic \& Geometric Topology 19}, 2 (2019), 657--700.

\bibitem{cirici_model_2019}
{\sc Cirici, J., Santander, D.~E., Livernet, M., and Whitehouse, S.}
\newblock {Model Category Structures and Spectral Sequences}.
\newblock {\em Proceedings of the Royal Society of Edinburgh Section A:
  Mathematics\/} (Aug. 2019), 1--34.

\bibitem{console_dolbeault_2001}
{\sc Console, S., and Fino, A.}
\newblock {Dolbeault Cohomology of Compact Nilmanifolds}.
\newblock {\em Transformation Groups 6}, 2 (June 2001), 111--124.

\bibitem{deligne_real_1975}
{\sc Deligne, P., Griffiths, P., Morgan, J., and Sullivan, D.}
\newblock {Real Homotopy Theory of K\"ahler Manifolds.}
\newblock {\em Inventiones mathematicae 29\/} (1975), 245--274.

\bibitem{fino_dolbeault_2019}
{\sc Fino, A., Rollenske, S., and Ruppenthal, J.}
\newblock {Dolbeault Cohomology of Complex Nilmanifolds Foliated in Toroidal
  Groups}.
\newblock {\em The Quarterly Journal of Mathematics 70}, 4 (Dec. 2019),
  1265--1279.

\bibitem{gabriel_unzerlegbare_1972}
{\sc Gabriel, P.}
\newblock {Unzerlegbare Darstellungen I.}
\newblock {\em Manuscripta mathematica 6\/} (1972), 71--104.

\bibitem{griffiths_principles_1978}
{\sc Griffiths, P., and Harris, J.}
\newblock {\em {Principles of {{Algebraic Geometry}}}}.
\newblock {Wiley}, Oct. 1978.

\bibitem{hirzebruch_topological_1978}
{\sc Hirzebruch, F.}
\newblock {\em {Topological Methods in Algebraic Geometry}}, vol.~131 of {\em
  {Grundlehren Der Mathematischen {{Wissenschaften}}}}.
\newblock {Springer}, {Berlin [et al.]}, 1978.

\bibitem{khovanov_spectral_2007}
{\sc Khovanov, M.}
\newblock {Spectral Sequences via Indecomposable Bicomplexes}.
\newblock unpublished note.

\bibitem{khovanov_faithful_2020}
{\sc Khovanov, M., and Qi, Y.}
\newblock {A {{Faithful Braid Group Action}} on the {{Stable Category}} of
  {{Tricomplexes}}}.
\newblock {\em SIGMA. Symmetry, Integrability and Geometry: Methods and
  Applications 16\/} (Mar. 2020), 019.

\bibitem{klingler_cherns_2017}
{\sc Klingler, B.}
\newblock {Chern's Conjecture for Special Affine Manifolds}.
\newblock {\em Annals of Mathematics 186}, 1 (July 2017), 69--95.

\bibitem{kuperberg_notitle_2010}
{\sc Kuperberg, G.}, May 2010.
\newblock \url{https://mathoverflow.net/questions/25723/}.

\bibitem{mchugh_bott-chern-eppli_2017}
{\sc McHugh, A.}
\newblock {{B}ott-{C}hern-{A}eppli and {F}rolicher on {C}omplex 3-folds}.
\newblock {\em arXiv:1708.03251 [math]\/} (Aug. 2017).

\bibitem{megy_three_2014}
{\sc M{\'e}gy, D.}
\newblock {Three Lectures on {{Hodge}} Structures}.
\newblock Spring school `Classical and p-adic Hodge theory', Rennes, 2014.

\bibitem{meng_heredity_2019}
{\sc Meng, L.}
\newblock {The heredity and bimeromorphic invariance of the
  $\partial\bar{\partial}$-lemma property}.
\newblock {\em arXiv:1904.08561 [math]\/} (Apr. 2019).

\bibitem{neisendorfer_dolbeault_1978}
{\sc Neisendorfer, J., and Taylor, L.}
\newblock {Dolbeault Homotopy Theory}.
\newblock {\em Transactions of the American Mathematical Society 245\/} (1978),
  183--210.

\bibitem{popovici_adiabatic_2019-2}
{\sc Popovici, D.}
\newblock {Adiabatic Limit and the {{Fr{\"o}licher}} Spectral Sequence}.
\newblock {\em Pacific Journal of Mathematics 300}, 1 (July 2019), 121--158.

\bibitem{qi_quiver_2010}
{\sc Qi, Y.}
\newblock {Quiver Representations and Spectral Sequences}.
\newblock \url{http://www.its.caltech.edu/~youqi/qrss.pdf}, Mar. 2010.

\bibitem{raghunathan_discrete_1972}
{\sc Raghunathan, M.~S.}
\newblock {\em {Discrete {{Subgroups}} of {{Lie Groups}}}}.
\newblock {Ergebnisse Der {{Mathematik}} Und Ihrer {{Grenzgebiete}}. 2.
  {{Folge}}}. {Springer-Verlag}, {Berlin Heidelberg}, 1972.

\bibitem{rao_dolbeault_2019}
{\sc Rao, S., Yang, S., and Yang, X.}
\newblock {Dolbeault cohomologies of blowing up complex manifolds}.
\newblock {\em Journal de Math\'ematiques Pures et Appliqu\'ees 130\/} (Oct. 2019), 68--92.

\bibitem{ringel_quivers_2012}
{\sc Ringel, C.~M.}
\newblock {Quivers and Their Representations: {{Basic}} Definitions and
  Examples}.
\newblock Lecture Notes, KUH Jeddah, 2012.

\bibitem{serre_theoreme_1955}
{\sc Serre, J.-P.}
\newblock {Un th\'eor\`eme de dualit\'e.}
\newblock {\em Commentarii mathematici Helvetici 29\/} (1955), 9--26.

\bibitem{speyer_notitle_2012}
{\sc Speyer, D.~E.}, Jan. 2012.
\newblock \url{https://mathoverflow.net/questions/86947/}.

\bibitem{stelzig_double_2018}
{\sc Stelzig, J.}
\newblock {\em {Double {{Complexes}} and {{Mixed Hodge Structures}} as {{Vector
  Bundles}}}}.
\newblock PhD thesis, WWU, {M\"unster}, 2018.

\bibitem{stelzig_double_2019}
{\sc Stelzig, J.}
\newblock {The {{Double Complex}} of a {{Blow}}-Up}.
\newblock {\em International Mathematics Research Notices\/} (July 2019).

\bibitem{wlodarczyk_toroidal_2003}
{\sc W{\l}odarczyk, J.}
\newblock {Toroidal Varieties and the Weak Factorization Theorem}.
\newblock {\em Inventiones mathematicae 154}, 2 (Nov. 2003), 223--331.

\bibitem{yang_bottchern_2020}
{\sc Yang, S., and Yang, X.}
\newblock {Bott\textendash{{Chern}} Blow-up Formulae and the Bimeromorphic
  Invariance of the {$\partial\partial$}-{{Lemma}} for Threefolds}.
\newblock {\em Transactions of the American Mathematical Society\/} (2020).

\end{thebibliography}

Jonas Stelzig, Mathematisches Institut der LMU, Theresienstraße 39, 80333 München, Jonas.Stelzig@math.lmu.de.
\end{document}